\documentclass{amsart}

\usepackage[utf8]{inputenc}
\usepackage{amssymb, amsmath,amscd}
\usepackage{amsthm}
\usepackage[]{tikz}
\usepackage[]{tikz-cd}
\usepackage{hyperref}
\usepackage[nameinlink, capitalize]{cleveref}
\usetikzlibrary{cd}
\usepackage{comment}

\usepackage{cobcat-macros}

\title{The Equivariant Cobordism Category}

\author{S{\o}ren Galatius}
\email{galatius@math.ku.dk}

\author{Gergely Sz\H ucs}
\email{yyegreg@gmail.com}

\begin{document}

\begin{abstract}

For a finite group $G$, we define an equivariant cobordism category $\mathcal{C}_d^G$. Objects of the category are $(d-1)$-dimensional closed smooth $G$-manifolds and morphisms are smooth $d$-dimensional equivariant cobordisms. We identify the homotopy type of its classifying space (i.e.\ geometric realization of its simplicial nerve) as the fixed points of the infinite loop space of a certain equivariant Thom spectrum.
    
\end{abstract}

\maketitle
\tableofcontents

\section{Introduction}\label{intro}
For each finite group $G$ we define the equivariant cobordism category $\mathcal{C}_d^G$. Objects are $(d-1)$-dimensional closed smooth manifolds equipped with a smooth action of $G$. The morphism space has homotopy type 
$$\mathcal{C}_d^G(M_0, M_1)\simeq \coprod\limits_L B\Diff^G (L, \partial L), $$
where $B\Diff^G(L, \partial L)$ denotes the classifying space of $\Diff^G(L, \partial L)$, the topological group of $G$-equivariant diffeomorphisms that fix the boundary pointwise, and the disjoint union is over $G$-manifolds $L$ equipped with an equivariant diffeomorphism $\partial L \cong M_0\coprod M_1$, one in each equivariant diffeomorphism class relative $\partial L$.

The main result of this paper identifies the homotopy type of the classifying space (the geometric realization of the nerve) as the fixed point space of the infinite loop space of a certain orthogonal $G$-spectrum.
\begin{thm}\label{thm:mainthm}
  There 
  is a weak equivalence
  \begin{equation*} 
    B\mathcal{C}_d^G \simeq \left ( \Omega^{\infty-1} \MTO_d \right )^G.
  \end{equation*}
\end{thm}
The genuine $G$-equivariant spectrum $\MTO_d$ is defined in Section~\ref{sec:orth-texorpdfstr-spe}, as a certain Thom spectrum.  As the special case $G={1}$ we recover the statement of \cite{gmtw} determining the homotopy type of the non-equivariant cobordism category.

A full description of the two spaces and the map in Theorem~\ref{thm:mainthm} is too technical for an introduction, but let us describe some aspects of it.  For a manifold $B$ with trivial action and a smooth closed $G$-manifold $L$, a smooth equivariant $L$-bundle is a smooth bundle $E\to B$ where $E$ is equipped with a smooth action of $G$ and the fibers are equivariantly diffeomorphic to $L$. 
Such bundles are classified by homotopy classes of maps $B\to B\Diff^G(L)$.
By definition of $\mathcal{C}_d^G$, one of the path components of $\mathcal{C}_d^G(\varnothing,\varnothing)$, namely the endomorphism monoid of the empty set, is a model for $B\Diff^G(L)$, resulting in a map $B\Diff^G (L) \hookrightarrow \mathcal{C}_d^G (\varnothing, \varnothing)$  and an induced map $\mathcal{C}_d^G (\varnothing, \varnothing)\to \Omega B\mathcal{C}_d^G$. 
For an equivariant $L$-bundle $E\to B$ classified by a map $f\colon B\to B\Diff^G(L)$, the homotopy class of the composite 
\begin{align}\label{classify}
    B\xrightarrow{f} B\Diff^G(L)\to \Omega B\mathcal{C}_d^G\to \left ( \Omega^{\infty} \MTO_d \right )^G  
\end{align} 
can now be described using an equivariant version of the Pontryagin-Thom construction as follows. 
Choose a fiberwise embedding 
\[
    \begin{tikzcd}
    E \arrow[rr, hook] \arrow[dr]& & B\times V \arrow[dl]\\
     & B & 
    \end{tikzcd} 
\]
for some sufficiently large $G$-representation $V$. 
This induces a map $E\to \mathrm{Gr}_d(V)$ classifying the vertical tangent bundle of $E$. If $\nu_v$ denotes the vertical normal bundle, we get a map of equivariant bundles 
\[
\begin{tikzcd}
\nu_v \arrow[r]\arrow[d] & \xi_V^{\bot} \arrow[d] \\
E \arrow[r] & \mathrm{Gr}_d(V),
\end{tikzcd}
\]
where $\xi_V^{\bot}$ denotes the complement of the tautological bundle. 
Writing $Th(\nu_v)$ and $Th(\xi_V^{\bot})=\MTO_d(V)$ for the Thom spaces of the bundles and composing with the Pontryagin collapse gives an equivariant map $B\times S^V\to Th(\nu_v)\to \MTO_d(V)$. 
The adjoint gives a map $B\to \left (\Omega^V \MTO_d(V) \right )^G$, since $B$ has trivial action. 
The space $\Omega^\infty \MTO_d$ is defined as the colimit of $\Omega^V \MTO_d(V)$ over the poset of finite dimensional subrepresentations of a universal representation $\mathcal{U}_G$. 
The composite $$B\to \left( \Omega^V \MTO_d(V) \right )^G \to \left (\Omega^\infty \MTO_d \right )^G$$ 
is the homotopy class of \eqref{classify}. 

After giving the necessary definitions in \cref{sec:defs}, we prove Theorem~\ref{thm:mainthm} in \cref{sec:first_deloop}, \cref{sec:general_deloop} and \cref{sec:affine_grassmanian}. 
In \cref{sec:tangential_structures}, we introduce a version of Theorem~\ref{thm:mainthm} with tangential structures. The input is a space $\Theta$ with commuting actions of $G$ and $\GL_d(\mathbf{R})$. 
For a $G$-manifold $W$, an equivariant $\Theta$-structure is a $G\times \GL_d(\mathbf{R})$-equivariant map from the frame bundle $\Fr(W)\to \Theta$. 
We define the category $\mathcal{C}_\Theta^G$ of $G$-manifolds with $\Theta$-structure, a corresponding spectrum $MT\Theta$, and discuss a generalization of Theorem~\ref{thm:mainthm} in this setting. 

Finally, we relate our result to classical notions of equivariant bordism groups in \cref{sec:equi_bordism}.

\subsection{Acknowledgements}\ 
\label{sec:ack}

GS: I would like to thank SG foremost for all the invaluable input throughout, as well as for agreeing to be a coauthor on this paper and helping with its publication. GS is also grateful for the constructive feedback provided by Sander Kupers and Jens Reinhold.

SG: This paper is based on GS's PhD thesis, and I wish to acknowledge that the better part of this work was carried out by him.  I thank him for his invitation to join this project as a coauthor (which, under other circumstances I would have declined and insist that he finish the project on his own).

SG would like to thank Andrew Blumberg and Mona Merling for helpful discussions about equivariant delooping.

Both authors were supported by the European Research Council (ERC) under the European Union’s Horizon 2020 research and innovation programme (grant agreement No 682922), as well as by NSF grant DMS-1405001.  SG was additionally supported by the EliteForsk Prize, and by the Danish National Research Foundation (DNRF92 and DNRF151).

\section{Definitions}\label{sec:defs}
First we briefly review the theory of smooth equivariant bundles. For a finite group $G$, we define equivariant versions $\psi_d(V,W)$ of spaces of manifolds analogous to those in \cite{monoids}, then
define an equivariant version $\mathcal{C}_d$ of the cobordism category such that the fixed point category $\mathcal{C}_d^G$ recovers the $G$-bordism category described informally in the introduction. 
We also give a careful definition of $\MTO_d$ as an orthogonal $G$-spectrum. Finally, we describe the equivariant map (or rather zig-zag of maps)
$B\mathcal{C}_d \to \Omega^{\infty-1} \MTO_d$, which, after taking fixed points becomes the equivalence in Theorem~\ref{thm:mainthm}.  (We shall leave it to the interested reader to relate this morphism in the equivariant homotopy category with the Pontryagin--Thom construction described in the introduction.)

    \subsection{Equivariant bundles}\label{sec:equivariant-bundles}
    We 
    recall some definitions and results about equivariant bundles. First we discuss the general theory, then focus on the case of smooth manifold bundles.  Let us emphasize that this section is not logically necessary for the proof of Theorem~\ref{thm:mainthm}, instead the goal is to motivate the definition of the equivariant cobordism category $\mathcal{C}_d$ in subsection~\ref{sec:embedd-cobord-categ} below, and to give an interpretation of the path components in some of its morphism spaces as classifying spaces for equivariant smooth manifold bundles.
    \begin{defn}
        Let $A$ be a topological group and $G$ a finite group. A $G\mhyphen A$-bundle is a fiber bundle $p \colon E \to B$ with structure group $A$, together with actions of $G$ on $E$ and $B$ such that $p$ is $G$-equivariant and $G$ acts through maps of principal $A$-bundles. This is equivalent to saying that in the associated principal $A$-bundle $\pi\colon P\to B$, the space $P$ is a $(G\times A)$-space and $\pi$ is $G$-equivariant.
    \end{defn}
    We will write the action of $G$ on $P$ from the left and the action of $A$ on the right.
    As discussed in \cite{lashof} and \cite{bierstone}, in order for $G\mhyphen A$-bundles to have the right homotopical properties, they need to satisfy a $G$-local triviality condition. 
    For $G$ finite and $B$ Hausdorff, this can be stated as follows.
    \begin{defn}[Bierstone's condition]\label{bierstone}
        Let $p\colon E\to B$ be a $G\mhyphen A$-bundle with fiber $F$. 
        For $b\in B$ write $G_b \leq G$ for the stabilizer of $b$.
        We say the bundle satisfies Bierstone's condition if for each $b\in B$ 
        there is a $G_b$-invariant neighborhood $U_b$ of $b$ in $B$ and $G_b$-equivariant map $p^{-1}(U_b) \to U_b \times F$ that is an equivalence of $A$-bundles over $U_b$. Here $U_b \times F$ has $G_b$-action given by
        $$
        h(u, y) = (hu, \rho_b(h)y),
        $$
        for some homomorphism $\rho_b\colon G_b \to A$, where $u \in U_b$, $h\in G_b$, $y\in F$.
    \end{defn}
    Note that the homomorphism $\rho$ above is determined up to conjugacy by the action of $G_b$ on the fiber $p^{-1}(b)$.  We say two principal $G\mhyphen A$-bundles $P_1$ and $P_2$ over $B$ are equivalent if there is a $G\times A$-equivariant homeomorphism $P_1\to P_2$ over $B$.

    A slightly stronger condition on a $G\mhyphen A$-bundle $p\colon E \to B$ is to be \emph{$G\mhyphen A$-locally trivial} which additionally requires trivializability with respect to a $G$-equivariant open cover by slices, see \cite[Definition on p.~258 and Lemma 1.1]{lashof}.  Finally there is the even stronger notion of being \emph{numerable}, in which the open cover by slices is required to admit a partition of unity, see \cite[Definition on p.~262]{lashof}.  If the base $B$ is a $G$-manifold, or more generally a topological space which is Hausdorff, paracompact and completely regular, then these three notions agree, see \cite[Lemma 1.3 and Corollary 1.13]{lashof}.  Therefore we shall not dwell on the distinction.

    As explained in \cite[Section 2]{lashof}, there exists a \emph{universal} numerable principal bundle $p: E_G A \to B_G A$, such that numberable $G\mhyphen A$-bundles over $B$ is in bijective correspondence with $[B, B_GA]^G$, the set of equivariant homotopy classes of equivariant maps, and the correspondence is given by pullback of $p$.
  The following theorem gives a characterization of universal bundles (see \cite[Theorem 2.14]{lashof}).
    
    \begin{proposition}\label{universal_criterion}
    A numerable principal $G\mhyphen A$-bundle $\pi\colon E_GA\to B_GA$ is universal if for each $H\leq G$ and each homomorphism $\rho\colon H\to A$, the fixed point space $E_GA^H$ is contractible, where $H$ acts via $z\mapsto hz\rho(h)^{-1}$ for $z\in E_GA$, $h\in H$.
    \end{proposition}
    This is related to the notion of classifying spaces for families of closed subgroups from equivariant homotopy theory as follows, see also \cite[VII.2]{alaska}.  Let  $\mathcal{F}$ be the family of closed subgroups $H \leq G \times A$ satisfying $H \cap (\{e\} \times A) = \{e\}$.  In other words $\mathcal{F}$ consists of those subgroups arising as $\{(h,\rho(h)) \mid h \in H\}$ for a subgroup $H \leq G$ and a homomorphism $\rho \colon H \to A$.
    There is then a universal $(G \times A)$-CW complex $E\mathcal{F}$ whose fixed points for any closed subgroups $H \leq G \times A$ satisfy
    \begin{enumerate}
    \item $E\mathcal{F}^H \simeq \ast$ for $H \in \mathcal{F}$,
    \item $E\mathcal{F}^H = \emptyset$ for $H \not\in \mathcal{F}$.
    \end{enumerate}
    In particular the group $A = \{e\} \times A$ itself acts freely and the quotient map
    \begin{equation*}
      E\mathcal{F} \to E\mathcal{F}/A
    \end{equation*}
    is a model for $E_G A \to B_G A$.

    Let us now consider the case $A = \Diff(M)$ for a smooth closed manifold $M$.  For a finite dimensional orthogonal $G$-representation $V$ and a closed smooth manifold $M$, let $\Emb(M,V)$ denote the space of embeddings $M\hookrightarrow V$ equipped with the $C^\infty$ topology.
    This space has an action of $G \times \Diff(M)$ where $\Diff(M)$ acts on the right by precomposition and $G$ on the left by postcomposition. 
    Choose a universal $G$-representation $\mathcal{U}_G$, i.e.\ an infinite dimensional representation containing an isomorphic copy of every finite dimensional representation, and let $$\Emb(M, \mathcal{U}_G)=\colim_{V\in s(\mathcal{U}_G)} \Emb(M,V),$$
    where $s(\mathcal{U}_G)$ denotes the poset of finite dimensional subrepresentations of $\mathcal{U}_G$.

    \begin{proposition}\label{B_GDiff}
      The $G$-space $\Emb(M, \mathcal{U}_G)/\Diff(M)$ is $G$-equivariantly weakly equivalent to the equivariant classifying space $B_G\Diff(M)$.  More precisely, there is a commutative diagram
      \begin{equation}\label{eq:14}
        \begin{tikzcd}
          E_G \Diff(M) \rar["\simeq_G"] \dar & \Emb(M, \mathcal{U}_G) \dar\\
          B_G \Diff(M) \rar["\simeq_G"] & \Emb(M, \mathcal{U}_G)/\Diff(M),
        \end{tikzcd}
      \end{equation}
      where the vertical maps are the quotient maps and the horizontal maps are $G$-equivariant weak equivalences.
    \end{proposition}

    \begin{proof}
      Let
      $H\leq G$ be a subgroup and $\rho\colon H\to \Diff(M)$ a homomorphism. Consider $\Emb(M, \mathcal{U}_G)$ under the left $H$-action defined by $\widetilde{\rho}\colon H \to G\times \Diff(M)$ given by $h\mapsto (h, \rho(h)^{-1})$. Then thefixed points $\Emb(M, \mathcal{U}_G)^H$ are $H$-equivariant embeddings $M\to V$ where the $H$-action on $M$ is given by $\rho$. Thus $\Emb(M, \mathcal{U}_G)^H$ is weakly contractible by the Mostow–Palais theorem (the equivariant analogue of Whitney embedding).  On the other hand, clearly $\Emb(M,\mathcal{U}_G)^H = \emptyset$ for any $H \leq G \times \Diff(M)$ containing an element of the form $(e,f)$ with $f \in \Diff(M) \setminus \{e\}$.  The general theory of classifying spaces for families then provides a $(G \times \Diff(M))$-equivariant weak equivalence
      \begin{equation*}
        E\mathcal{F} \to \Emb(M,\mathcal{U}_G),
      \end{equation*}
      and by taking quotients by $\Diff(M)$ we obtain the commutative diagram~\eqref{eq:14}.

      It remains to see that the bottom horizontal map in~\eqref{eq:14} is an equivariant weak equivalence.  Firstly, it is easy to see that the diagram~\eqref{eq:14} is \emph{cartesian}: for any $x \in B_G\Diff(M)$ with stabilizer $H \leq G$, the induced map of vertical fibers is an $H$-equivariant homeomorphism.  That is because both fibers are torsors for the centralizer
      \begin{equation*}
        C_{\Diff(M)}(\rho(H)),
      \end{equation*}
      where $\rho\colon H \to \Diff(M)$ is a homomorphism whose conjugacy class is determined by $x$.  The proof will now be finished if we can show that the maps
      \begin{equation}\label{eq:16}
        (\Emb(M, \mathcal{U}_G))^H \to (\Emb(M, \mathcal{U}_G)/\Diff(M))^H
      \end{equation}
      are quasi-fibrations for all $H \leq G$.  Indeed, the long exact sequences and the five-lemma then imply that the induced maps of homotopy groups
      \begin{equation*}
        \pi^H_n(B_G \Diff(M)) \to \pi_n^H(\Emb(M, \mathcal{U}_G)/\Diff(M))
      \end{equation*}
      are isomorphisms for all subgroups $H \leq G$, all $n \geq 1$, and all basepoints; the case $n=0$ is easily handled separately (the map~\eqref{eq:16} induces a surjection on $\pi_0^H$).

      To see that~\eqref{eq:16} is a quasi-fibration we first consider $V \in s(\mathcal{U}_G)$.  By the same argument as \cite[Proposition 5]{binz-fischer} we deduce that the quotient map $\Emb(M,V) \to \Emb(M,V)/\Diff(M)$ has $H$-equivariant local sections near any $H$-fixed point.  (By using the $H$-invariant Riemannian metric on $V$, the tubular neighborhood denoted $S^\epsilon_i$ in op.cit.\ can be chosen $H$-invariant.  The rest of the argument then applies verbatim.)  This is equivalent to satisfying Bierstone's condition, by \cite[Lemma 1.4]{lashof}, which implies that the induced map
      \begin{equation}\label{eq:18}
        (\Emb(M,V))^H \to (\Emb(M,V)/\Diff(M))^H
      \end{equation}
      is a fiber bundle, and in particular a Serre fibration.  In this situation, taking homotopy groups commutes with the filtered colimit over $V \in s(\mathcal{U}_G)$, since any inclusion $V \subset V'$ is sent to an inclusion of a closed subspaces.  Since the homotopy groups of~\eqref{eq:18} and of its point-set fibers fit into a long exact sequence for every $V$, the same holds in the colimit, proving that~\eqref{eq:16} is a quasi-fibration.
    \end{proof}        
    
    To relate this model of $B_G\Diff(M)$ with the discussion in Section~\ref{intro} we pass to $G$-fixed point spaces.  Up to weak equivalence they may be described in terms of groups of equivariant diffeomorphisms, as follows (proved by the argument of \cite[Theorem VII.2.4]{alaska}).
    \begin{lemma} For any subgroup $H\leq G$
          $$\big ( B_G(\Diff(M)) \big )^H\simeq \coprod\limits_{\rho} B\Diff^H (M,\rho), $$ 
          where the disjoint union is over conjugacy classes of homomorphisms $\rho\colon H \to \Diff(M)$ and $\Diff^H(M,\rho)$ is the group of equivariant diffeomorphisms, that is, the centralizer $C_{\Diff(M)}(\rho(H))$ of the image under $\rho$.
    \end{lemma}

    As in the non-equivariant case, it is sometimes more natural to consider \emph{smooth} bundles, defined as follows.
    \begin{defn}
      Let $M$ be a smooth closed manifold.  A smooth $G$-equivariant bundle with fiber $M$ is a $G\mhyphen \Diff(M)$-bundle $p: E \to B$ with fiber $M$ with $E$ and $B$ smooth $G$-manifolds, such that Bierstone's condition is satisfied and that the local trivializations may be chosen smooth.

      A smooth $G$-equivariant manifold bundle is a disjoint union of such, over varying $M$.
    \end{defn}
    We also have an analogue of Ehresmann's fibration theorem (see \cite[1.12]{ulrich} for a proof).
    \begin{lemma}[Ehresmann's lemma]
        If $p\colon E\to B$ is a $G$-equivariant proper submersion of smooth $G$-manifolds, then it is a $G$-equivariant manifold bundle.
    \end{lemma}
    Smooth bundles in this sense have the same classification theory as principal $G\mhyphen\Diff(M)$-bundles, in the sense that isomorphism classes of smooth $G$-equivariant bundles with fiber $M$ over some smooth $G$-manifold $B$ are in natural bijection with the set $[B,B_G\Diff(M)]^G$ of equivariant homotopy classes of maps $B \to B_G \Diff(M)$.  The point is that continuous equivariant maps $B \to \Emb(M,\mathcal{U}_G)/\Diff(M)$ may be approximated by smooth ones, and that we may then pull back the ``universal'' smooth $G$-equivariant bundle with fiber $M$, modeled as
    \begin{equation*}
      \frac{\Emb(M,\mathcal{U}_G) \times M}{\Diff(M)} \stackrel{p}{\longrightarrow} \frac{\Emb(M,\mathcal{U}_G) }{\Diff(M)}.
    \end{equation*}
    This finishes our discussion of how to interpret the quotient $\Emb(M,\mathcal{U}_G)/\Diff(M)$ as a classifying space for $G$-equivariant smooth bundles.

    \subsection{Spaces of manifolds}
    \begin{defn}\label{def:Psi-d-V}
      For any finite dimensional inner product space $V$ and open subset
      $O \subset V$, let $\Psi_d(O)$ be the set of closed subsets $M\subset O$ which are smooth $d$-dimensional (not necessarily compact) manifolds without boundary. Consider these as spaces with the topology defined in \cite[Section 2.1]{monoids}, which is a $C^\infty$ variant of the compact-open topology (see \cite[Section 2]{chris_sch} for another approach to defining the topology). 
    \end{defn}
    
    The topology on $\Psi_d (O)$ has the following property. For a submersion of manifolds $\pi\colon E\to B$ and a proper embedding $\iota$ over B
    \[
    \begin{tikzcd}
        E \arrow[rr, hook, "\iota"] \arrow{dr}{\pi} &  &B\times O \arrow{dl}{} \\
        & B &
    \end{tikzcd}
    \]
    the associated map $f\colon B\to \Psi_d(O)$ given by $b\mapsto \{b\}\times O \cap \iota(E)$ is continuous. We call the map $f\colon B\to \Psi_d(O)$ smooth in the above case, and $E$ is the graph of $f$.

    If $\varphi\colon O'\to O$ is an open embedding then we have a continuous map $\Psi_d(O)\to \Psi_d(O')$, mapping $M$ to its inverse
    image under $\varphi$. In particular if $G$ is a finite group, $V$ an orthogonal $G$-representation, and $O \subset V$ a $G$-invariant open subset, we get an action of $G$ on $\Psi_d(O)$ where fixed points $\Psi(O)^G$ are sets of $G$-manifolds $M$ equivariantly embedded in $O$ as a closed subset. 
    
     Now let $B$ be a manifold with trivial action, $E$ a $G$-manifold with an equivariant submersion $\pi\colon E\to B$, $V$ a finite dimensional orthogonal $G$-representation, and $O \subset V$ a $G$-invariant open subset. If we have an equivariant embedding $\iota$ as above, we call the associated continuous map $B\to \Psi_d(O)^G$ smooth. 
     
     \begin{lemma}\label{smoothapprox}
         Let $V$ be a $G$-representation, $O\subset V$ a $G$-invariant open subset. Let $B$ be a smooth manifold and let $f\colon B\to \Psi_d(O)^G$ be a continuous map. Let $S\subset B\times O$ be open, and $T\subset B\times O$, both invariant subsets such that $\overline{S}\subset \textup{int}(T)$. Then there exist a homotopy $F\colon [0,1]\times B \to \Psi_d(O)^G$ starting at $f$, which is smooth on $(0,1]\times S\subset [0,1]\times B \times O$ and is constant outside $T$. Furthermore, if $f$ is already smooth on an open set $A\subset S$ then the homotopy can be assumed to be smooth on $[0,1]\times A$.
     \end{lemma}
    
    \begin{proof}
    The proof is analoguous to \cite[Lemma 2.17]{monoids}, we point out the observations needed to address the equivariant case.
        Following \cite[Definition 2.1]{monoids}, the topology on $\Psi_d(V)^G$ can be built as a limit from the compactly supported topology $\Psi_d(V)^G_{cs}$, which is an infinite dimensional manifold, 
     modelled on the vector spaces $\Gamma_c^G (NM)$ of compactly supported equivariant sections of the normal bundle at a point $M\in \Psi_d(V)^G$. 
     Note that the topology defined this way agrees with the subspace topology $\Psi_d(V)^G\subset \Psi_d(V)$.
    \end{proof}
    
    \begin{defn}
    For $W$ a subspace of $V$, let $V-W$ denote the orthogonal complement, and let $D_1(V-W)$ be the open unit disc. Then $\psi_d(V,W)$ is the subspace of $\Psi_d(V)$ consisting of $M$ such that $M\subset D_1(V-W)\times W$. When $V$ is a $G$-representation and $W$ a subrepresentation, $\psi_d(V,W)$ inherits an action of $G$.
    \end{defn}
    
    By a minor abuse of notation, the definition above identifies $M \subset V$ with its image under the canonical isomorphism $V \cong (V - W) \times W$.  Similar abuses will occur later.

\begin{lemma}
    There is an equivariant homotopy equivalence 
    $$\psi_d(\mathcal{U}_G, 0)=
        \colim_{V\in \mathcal{U}_G}\psi_d(V, 0)\simeq \coprod\limits_M B_G(\Diff(M)), $$
        where the disjoint union is over closed smooth $d$-dimensional manifolds, one in each diffeomorphism class.
\end{lemma}
\begin{proof}
    The space $\psi_d(V,0)$ is homeomorphic to
    \[
    \coprod\limits_M \Emb(M,V)/\Diff(M), 
    \]
    so by \cref{B_GDiff} the claim follows.
\end{proof}

    \subsection{The embedded cobordism category}\label{sec:embedd-cobord-categ}
    First
    we define the embedded cobordism category $\mathcal{C}_d(V)$ for $V$ a finite dimensional orthogonal $G$-representation. This is a category with strict $G$-action, i.e.\ for any $g\in G$ we have a functor $\mathcal{C}_d(V) \to \mathcal{C}_d(V)$, and composition gives equal functors.
    \begin{defn}\label{embedded_cobcat}
    For a finite dimensional orthogonal $G$-representation $V$, the topological category $\mathcal{C}_d(V)$ has object space
        $$\mathrm{Ob}(\mathcal{C}_d(V))=\psi_{d-1} (V,0). $$
        Morphisms are pairs
        $$(N,r)\in \mathcal{C}_d(M_1,M_2)\subset \psi_d(V\oplus 1, 1)\times \mathbf{R}, $$
        where $r\in \mathbf{R}_{> 0}$ and
        $N\in \psi_d(V\oplus 1, 1)$ is a manifold with 
        $N \cap (V\times (-\infty, \epsilon)) = M_1\times (-\infty, \epsilon)$, and
        $N \cap (V\times (r-\epsilon,+\infty)) = M_2\times (r-\epsilon,+\infty)$
        for some $\epsilon>0$.

        For
        fixed $\epsilon > 0$ this leads to a non-unital topological category $\mathcal{C}_d^\epsilon$ whose morphism spaces are topologized as subspaces of $\psi_d(V \oplus 1,1) \times \mathbf{R}$, and $\mathcal{C}_d$ itself is defined as the colimit of $\mathcal{C}_d^\epsilon$.
         
        The (strict) action of $G$ on $\mathcal{C}_d(V)$ comes from the action on $V$. Composition is given by concatenation.
    \end{defn}
    
    For  an equivariant isometric embedding of representations $V\to W$, we have a continuous equivariant functor $\mathcal{C}_d(V)\to \mathcal{C}_d(W)$, which on objects is given by the inclusion $\psi_{d-1} (V,0) \to \psi_{d-1} (W,0)$ and on morphisms given by the inclusion $\psi_d(V\oplus 1, 1)\to \psi_d(W\oplus 1, 1)$.

    The precise definition of the cobordism category $\mathcal{C}_d^G$ informally introduced in \cref{intro} is now as follows.

\begin{defn}\label{equivariant_cobordism_category}
  Choose a universal representation $\mathcal{U}_G$, and let $$\mathcal{C}_d(\mathcal{U}_G)=\colim_{\mathcal{U}_G} \mathcal{C}_d(V), $$
  where $\colim_{\mathcal{U}_G}$ denotes the colimit taken over the poset of finite dimensional subrepresentations of $\mathcal{U}_G$. 
  
  Finally, let 
  \begin{align*}
    \mathcal{C}_d^G= \left ( \mathcal{C}_d(\mathcal{U}_G) \right ) ^G
  \end{align*}
  be the fixed category.
\end{defn}

\begin{rmrk}\label{rmrk:non-unital}
  The cobordism category defined above is only a \emph{non-unital category}
    and its nerve is therefore only a semi-simplicial space, i.e.,\ without degeneracy maps.  The classifying space in Theorem~\ref{thm:mainthm} therefore denotes the ``thick'' geometric realization of the nerve.

  For later use, let us remark that our equivariant bordism category does have \emph{weak units}, namely the morphisms $(N,r)$ where $N = M \times \R$ is a cylinder: the maps of morphisms spaces defined by composing with a cylinder from the left or from the right are canonically homotopic to the identity.
\end{rmrk}
                
\subsection{The orthogonal \texorpdfstring{$G$-spectrum $\MTO_d$}{G-spectrum MTOd}}
\label{sec:orth-texorpdfstr-spe}
Let
$G$-$\mathrm{Top}$ denote the category whose objects are compactly generated topological $G$-spaces and whose morphisms are based equivariant maps and let $G$-$\mathrm{Top}_*$ be the corresponding based category.  Cartesian product and smash
product give these categories symmetric monoidal structures, so it makes sense to enrich over them.  In particular there is a $G$-$\mathrm{Top}_*$-enriched category $\mathrm{Top}_G$ whose objects are based compactly generated $G$-spaces and whose morphisms are all based (not necessarily equivariant).

Let $\mathcal{L}_G$ denote the category of finite dimensional orthogonal $G$-representations and isometric embeddings, enriched over $G$-$\mathrm{Top}$.

Following 
\cite[Section II.4]{MandellMay}, let $\mathcal{J}_G$ be the $G$-$\mathrm{Top}_*$-enriched category with 
$$\mathrm{Ob}(\mathcal{J}_G)=\mathrm{Ob}(\mathcal{L}_G)$$ 
and morphisms are the Thom space
 $$\mathcal{J}_G(V,W)=\mathrm{Th}\left ( 
 \begin{tikzcd} 
 (\textup{im}\varphi)^\bot \arrow{d}{} \\
 \mathcal{L}_G(V,W)
 \end{tikzcd} \right )
 $$
 where $(\textup{im}\varphi)^\bot$ denotes the vector bundle whose total space is $\{(\varphi, w) \in \mathcal{L}_G(V,W) \times V \mid w\in (\textup{im}\varphi)^\bot\}$. An orthogonal $G$-spectrum then is a $G$-$\mathrm{Top}_*$-enriched functor $\mathcal{J}_G\to \mathrm{Top}_{G}$. 
 
 \begin{defn}\label{MTOspectrum}
   The 
   orthogonal $G$-spectrum $\MTO_d$ is the enriched functor $\mathcal{J}_G\to \mathrm{Top}_G$ defined on objects as
 $$
 \MTO_d(V) = \mathcal{J}_G(\mathbf{R}^d, V) / O(d) \cong \mathrm{Th} \left (
  \begin{tikzcd} 
 \xi_V^\bot \arrow{d}{} \\
 \mathrm{Gr}_d(V) 
 \end{tikzcd} 
 \right ) ,
 $$
 where $\mathrm{Gr}_d(V)$ denotes the Grassmannian of $d$-planes in $V$, and $\xi^\bot$ is the orthogonal complement of the tautological bundle.
 
 Explicitly, a non-basepoint morphism $(\varphi, w)\in \mathcal{J}_G(V,W)$ is assigned the map $\MTO_d(V)\to \MTO_d(W)$ induced by Thomifying the bundle map
 \[
 \begin{tikzcd}
 \xi_V^\bot \arrow{d}{} \arrow{r}{\varphi(-)+w} & \xi_W^\bot \arrow{d}{} \\
 \mathrm{Gr}_d(V) \arrow{r}{\mathrm{Gr}_d(\varphi)} & \mathrm{Gr}_d(W).
 \end{tikzcd}
 \]
 \end{defn}

 \begin{defn}
    For an orthogonal $G$-spectrum $E$, and a universal representation $\mathcal{U}_G$ of $G$, let $s(\mathcal{U}_G)$ denote the poset of finite dimensional subrepresentations of $\mathcal{U}_G$. Then define
    $$
    \Omega^{\mathcal{U}_G} E = \colim\limits_{V\in s(\mathcal{U}_G)} \Omega^V E(V), 
    $$
    where $\Omega^V E(V)=\Map(S^V, E(V))$ is the $G$-space of pointed maps from the representation sphere. When the group and the universe is given in the context, we will write $\Omega^\infty$ instead of $\Omega^{\mathcal{U}_G}$.
 \end{defn}
 
 \begin{defn}
    For an orthogonal spectrum $E$ and a $G$-representation $V$, let $\sh_V E$ be the orthogonal spectrum given by $\sh_V E (W)=E(V\oplus W)$. Then denote $\Omega^{\mathcal{U}_G-V} E=\Omega^{\mathcal{U}_G} \sh_V E$
 \end{defn}

 \subsection{Proof of the main theorem}

 The
  equivalence in Theorem~\ref{thm:mainthm} will be deduced as a consequence of the following lemmas. 
\begin{proposition}\label{cobcat}
  There is an equivariant weak equivalence 
  $$B\mathcal{C}_d (V) \simeq \psi_d(V\oplus 1, 1).$$
\end{proposition}

\begin{proposition}\label{delooplemma}
  When $\dim(V^G)\geq d$, there is an equivariant weak equivalence 
  $$\psi_d(V\oplus 1, 1) \xrightarrow{\sim} \Omega^V \psi_d (V\oplus 1, V\oplus 1).$$
\end{proposition}

\begin{proposition}\label{affine_grassmanian}
  The is an equivariant weak equivalence 
  $$\MTO_d(V) \xrightarrow{\sim} \psi_d(V,V).$$
\end{proposition}

These propositions will be proved in \cref{sec:first_deloop}, \cref{sec:general_deloop} and \cref{sec:affine_grassmanian} respectively.  The following is now an immediate consequence.

\begin{thm}\label{unstable_theorem}
The maps defined above result in a zig-zag of equivariant weak equivalences
\[
    B\mathcal{C}_d (V) \simeq \psi_d(V\oplus 1, 1) \xrightarrow{\sim} \Omega^V \psi_d (V\oplus 1, V\oplus 1)
    \xleftarrow{\sim} \Omega^V \MTO_d(V\oplus 1),
\]
for any orthogonal $G$-representation $V$ with $\dim(V^G) \geq d$. \qed
\end{thm}         
\begin{thm}\label{centralthm}
 The maps above induce an equivariant equivalence
    $$
    B\mathcal{C}_d(\mathcal{U}_G) \simeq \Omega^{\mathcal{U}_G-1} \MTO_d.
    $$
\end{thm}

\begin{proof}
    The equivalences in \cref{unstable_theorem} are compatible under isometric embeddings $V\to W$, and hence taking colimits we get the equivalence in our theorem.
\end{proof}

The equivariant weak equivalence in \cref{centralthm} implies our Theorem~\ref{thm:mainthm} by taking fixed points.

\section{The classifying space of the equivariant cobordism category}\label{sec:first_deloop}
As the first step in the proof of \cref{centralthm} we show \cref{cobcat}: that for any $G$-representation $V$ there is an equivariant equivalence 
\begin{align}\label{htypecob}
    B\mathcal{C}_d (V)\simeq \psi_d(V\oplus 1, 1).
\end{align} 
The proof of this is very similar to the non-equivariant case (\cite{monoids}, Section 3). We outline the steps of the proof.
In \cref{models} we introduce the posets $\mathcal{D}_d(V)$ and $\mathcal{D}^\epsilon_d(V)$ and functors (of non-unital categories)
$$\mathcal{C}_d(V)\xleftarrow{} \mathcal{D}^\epsilon_d(V) \xrightarrow{} \mathcal{D}_d(V)$$
that induce level-wise equivariant equivalences of nerves, as shown in \cref{outlines}. Finally we show that the forgetful map $B\mathcal{D}_d(V)\to \psi_d(V\oplus 1, 1)$ is an equivariant equivalence.  Recall that we work with thick geometric realization, so level-wise equivalence induces equivalence of the realizations.

\subsection{Models for the equivariant cobordism category}\label{models}

Let $V$ be a finite dimensional $G$-representation. The following definitions agree with \cite[Definition 3.8]{monoids} and \cite[Theorem 3.9]{monoids} when $G = \{e\}$ and $V = \R^n$.

\begin{defn}
  Let $\mathcal{D}_d(V)$ be the following topological poset. Objects are pairs $(M,a)$ such that $M\in \psi_d(V\oplus 1,1)$ and $a\in \mathbf{R}$ is a regular value of the map $M\to \mathbf{R}$ induced by the projection $\pi_1\colon V\oplus 1 \to 1$. It is given the subspace topology $\mathrm{Ob}(\mathcal{D}_d(V))\subset \psi_d(V\oplus 1,1)\times \mathbf{R}$.

  The ordering is defined by $(M,a) < (M',a')$ if and only if $M = M'$ and $a < a'$. \end{defn}

We say that $M\in \psi_d(V\oplus 1,1)$ is cylindrical in the interval $(a,b)$ if there is $N\in \psi_{d-1}(V)$ such that 
\[
M\cap \pi_1^{-1}(a,b) = N\times (a,b). 
\]

\begin{defn}
  For $\epsilon > 0$, let $\mathcal{D}^\epsilon_d(V)$ be the topological poset defined similarly to $\mathcal{D}_d(V)$, but with objects the subspace of pairs $(M,a)\in \psi_d(V\oplus 1,1)\times \mathbf{R}$ such that $a$ is a regular value of $M\to \mathbf{R}$ and $M$ is cylindrical in $(a-\epsilon, a+\epsilon)$.
  \end{defn}

When $V$ is a $G$-representation, the posets $\mathcal{D}^\epsilon_d(V)$ and $\mathcal{D}_d(V)$ become posets with $G$-action. There is a natural inclusion of posets $i\colon\mathcal{D}^\epsilon_d(V)\to \mathcal{D}_d(V)$.

\begin{defn}
    Define a natural functor $p\colon \mathcal{D}^\epsilon_d(V)\to \mathcal{C}_d(V)$ as follows. On objects it maps $(M,a)$ to $\pi_1^{-1}(a)\cap M\in \psi_{d-1}(V)$. A morphism $(M,a) < (M,a')$ gets mapped to 
    $$
    \big((-\infty,0]\times (\pi_1^{-1}(a)\cap M) \big)
    \cup \big( \pi_1^{-1}(a,a')\cap M-a\mathbf{e}_1 \big)
    \cup \big( [0,\infty) \times (\pi_1^{-1}(a')\cap M) \big),  $$
    where $\mathbf{e}_1$ denotes the unit vector in the direction of the trivial representation $1$.
\end{defn}

\subsection{Equivalence of models}\label{outlines}

\begin{lemma}\label{314}
  The functor 
  $i\colon \colim_{\epsilon \to 0} \mathcal{D}^\epsilon_d(V)\to \mathcal{D}_d(V)$ induces a level-wise equivariant equivalence on the nerves.
\end{lemma}

\begin{proof}
  We
  prove this by the method of \cite[Lemma 3.4]{monoids} and refer there for more details.  We briefly recall the construction given there, in order to observe that it is compatible with $G$-action.

  Pick once and for all a smooth function $\lambda: \R \to \R$ with $\lambda(s) = 0$ for $|s| \leq 1$ and $\lambda(s) = s$ for $|s| \geq 2$.  Suppose given a point $x = (M, a_0, \dots, a_p) \in N_p D_d(V)^H$ for which $a_i > a_{i-1} + 2\epsilon$ for all $i = 1, \dots, p$.  Then for $t \in [0,1]$ we let
  \begin{equation*}
    f_{t,x}(s) =
    s + t \sum_{i = 0}^p (\epsilon \lambda(\frac{s - a_i}\epsilon) - s).
  \end{equation*}
  This defines a smooth homotopy from the identity to a non-decreasing function $\R \to \R$ which is constant near each $a_i$.  Then
  \begin{equation*}
    t \mapsto \big((\mathrm{Id}_V \times f_{t,x})^{-1} M,a_0, \dots, a_p)
  \end{equation*}
  defines a path from $x$ to a point in $\mathcal{D}^\epsilon_d(V)$.  It is clear that the entire path consists of points fixed by $H < G$ if $x$ is fixed by $H$.  It also depends continuously on $(t,x)$, so we have almost defined a deformation retraction.  The only problem is that this construction only works on the open subspace where the $a_i$'s stay apart by at least $2\epsilon$.
  
  Suppose now given a lifting problem
  \begin{equation*}
    \begin{tikzcd}
        \partial D^k \arrow[d, hook] \arrow[r] & \colim_\epsilon N_p D_d^\epsilon(V)^H \arrow[d, hook]\\
        D^k \arrow[r,"f"] \arrow[ur, dashed] & N_p D_d(V)^H
    \end{tikzcd}
  \end{equation*}
  for some $k  \in \mathbf{N}$.  By compactness of $\partial D^k$ the top map factors through some finite $\epsilon$, and by compactness of $D^k$ we may arrange that the $a_i$'s stay at least $2\epsilon$ apart in all values of $f$.  The homotopy described above therefore shows that the relative homotopy groups vanish.
\end{proof}

\begin{lemma}\label{315}
     The functor $p\colon \colim_{\epsilon \to 0} \mathcal{D}^\epsilon_d(V)\to \mathcal{C}_d(V)$ induces a level-wise equivariant equivalence on the nerves.   
\end{lemma}

\begin{proof}
  An equivariant homotopy inverse on simplicial nerves is given by the inclusion $N_p \mathcal{C}_d(V) \to \colim_\epsilon N_p \mathcal{D}^\epsilon_d(V)$ which sends $((M_1,a_1), \dots, (M_p,a_p))$ to the point $(M,0,a_1,\dots, a_p)$, where $(M,a) = (M_1,a_1) \circ \dots \circ (M_p,a_p)$ denotes the composition in $\mathcal{C}_d(V)$.

  The composition $N_p \mathcal{C}_d(V) \to \colim_\epsilon N_p \mathcal{D}^\epsilon_d(V) \to N_p \mathcal{C}_d(V)$ is the identity, and identifies $N_p \mathcal{C}_d(V)$ with the subspace of $\colim_\epsilon N_p \mathcal{D}^\epsilon_d(V)$ consisting of $(M,a_0,\dots, a_p)$ satisfying that $a_0 = 0$ and that $M$ is cylindrical outside $[0,a_p]$.  A deformation retract to that subspace may be defined as in the proof of \cite[Theorem 3.9]{monoids}.  We refer there for a more formal description of the homotopy, but let us describe it in words in order to convince 
  ourselves that it is an equivariant homotopy.  From a general $(M,a_0,\dots, a_p)$ we first parallel translate in the chosen $G$-fixed $\R$-direction so that $a_0$ becomes 0 (and $a_p$ becomes the old $a_p - a_0$), then push the parts of $M$ contained in $\pi_1^{-1}((-\infty,0])$ and $\pi_1^{-1}([a_p,\infty))$ off to $\pm \infty$ so that $M$ becomes cylindrical outside $[0,a_p]$.
\end{proof}

\begin{lemma}\label{316}
     The forgetful map $u\colon B\mathcal{D}_d(V)\to \psi_d(V\oplus 1, 1)$ taking a point 
     $$(M, a_0,\ldots, a_p, t_0,\ldots, t_p)\in N_p\mathcal{D}_d(V)\times \Delta^p \subset B\mathcal{D}_d(V)$$
     to $M\in \psi_d(V\oplus 1, 1)$ is an equivariant weak equivalence.
\end{lemma}

\begin{proof}
    The proof is done by showing that for any subgroup $H\leq G$ and any $q\in \mathbf{N}$, we can solve the following lifting problem.
    \[
    \begin{tikzcd}
        \partial D^q \arrow[r] \arrow[d] & BD_d(V)^H \arrow[d] \\
        D^q \arrow[r] \arrow[ru, dashed] & \psi_d(V\oplus 1, 1)^H
    \end{tikzcd}
    \]
    The proof of \cite[Theorem 3.10]{monoids} applies, since it only involves choices of regular values in the trivial summand.
\end{proof}

Lemmas \ref{314}, \ref{315} and \ref{316} now imply that there is a zig-zag of equivariant equivalences
\[
B\mathcal{C}_d(V)\xleftarrow{p} B\mathcal{D}^\epsilon_d(V) \xrightarrow{i} B\mathcal{D}_d(V)\xrightarrow{u} \psi_d(V\oplus 1, 1), 
\]
proving \eqref{htypecob}.

\section{Delooping}\label{sec:general_deloop}
The main step in showing \cref{centralthm} will be the proof of \cref{delooplemma} given in this section.  To establish the weak equivalence we first write down a map.

\begin{defn}
    Let $V$ a finite dimensional orthogonal representation of $G$, and let $W$ and $R$ be subrepresentations of $V$ that are orthogonal to each other. 
    Define equivariant maps
\begin{align} \label{looping}
    \psi_d(V,W)\xrightarrow{\alpha_R} \Omega^R \psi_d (V, W+R)
\end{align}
as follows. For $M\in \psi_d (V,W)$, $\alpha_R (M)$ is given by
\begin{align*}
    r\in R  & \mapsto M+r\in \psi_d (V,W+R) \\
    \infty & \mapsto \varnothing\in \psi_d (V,W+R),
\end{align*} 
which defines a continuous, based, and equivariant map $R\cup \{\infty \} =S^R \to \psi_d (V,W+R)$.
\end{defn}

The maps $\alpha_R(M)$ are continuous because of the compact-open nature of the topology on $\psi_d (V,W+R)$.  These maps are compatible in the sense that if $W$, $R$ and $R'$ are pairwise orthogonal subrepresentations of $V$, then the following diagram commutes
\[
\begin{tikzcd}
    \psi_d(V,W) \arrow[r, "\alpha_R"] \arrow[dr, "\alpha_{R\oplus R'}"'] &\Omega^R \psi_d(V,W+R) \arrow[d, "\Omega^R \alpha_{R'}"] \\
    & \Omega^{R\oplus R'} \psi_d(V,W+R+R').
\end{tikzcd}
\]

We 
show that \eqref{looping} is an equivariant equivalence, first in the case when $R$ is trivial, and then in \cref{nontriv} consider non-trivial $R$, assuming $W$ contains enough trivial summands.  Equivariant delooping by a trivial representation is not much different from the corresponding non-equivariant steps in \cite{monoids}, we outline this in subsection \ref{sec:triv-repr}.  Delooping by a non-trivial representation is harder, and will again use bar construction methods.  The assumption $\dim(W^G) > d$ is used in Lemma~\ref{forgetful_lemma} to see that certain bar constructions have the expected equivariant homotopy types.  It is then used again in Lemma~\ref{lemma-where-high-dim-is-used} to see that a certain monoid is group-like (in the equivariant sense, that the fixed-point monoid for any subgroup of $G$ is group-like).

    \subsection{Trivial representation}\label{sec:triv-repr}
    First we show that \eqref{looping} is an equivariant equivalence when $R$ is a one dimensional trivial representation. The proof is very similar to the non-equivariant case, although more care is needed to treat connected components. 
    
    Throughout this section $W$ is a trivial representation, $\dim(W)\geq 1$ and $R$ is a fixed trivial one dimensional subrepresentation of $V$ orthogonal to $W$. 
    Choose a unit vector $\mathbf{e}_R$ spanning $R$, and let $\pi_R\colon V\to \mathbf{R}$ denote projection onto $R\cong \mathbf{R}$ (so $\pi_R(\mathbf{e}_R) = 1$). 
    
    Similarly to \cref{models}, we start by introducing topological monoid and poset models $\mathcal{M}_d(V,W)$ and $\mathcal{P}_d(V,W)$ of $\psi_d(V,W)$, and our statement will follow from equivariant equivalences 
    $$
    B\mathcal{M}_d(V, W) \xleftarrow{\simeq} B\mathcal{P}_d (V, W) \xrightarrow{\simeq} \psi_d(V, W+R)
    $$
    and 
    $$
    \psi_d(V, W) \xrightarrow{\simeq} \mathcal{M}_d(V, W) \xrightarrow{\simeq} \Omega B \mathcal{M}_d(V, W).
    $$
    \begin{defn}
      Let $\mathcal{M}_d(V,W)$ be the topological monoid whose space of elements is the subspace
      $\mathcal{M}_d(V,W)\subset \psi_d(V,W+R)\times \mathbf{R}_{> 0}$ consisting of pairs $(M,a)$ such that $M\subset \pi_R^{-1} ((0,a))$. The composition of $(M,a)$ and $(M',a')$ is given by $(N,b)$ where 
        \[ N=M\cup (M'+a\mathbf{e}_R)\quad and  \quad b=a+a'. \]
    \end{defn}
    
    We have an injective map $\psi_d(V,W)\hookrightarrow \mathcal{M}_d(V,W)$ taking $M\in \psi_d(V,W)$ to  $(M+\mathbf{e}_R,2)\in \mathcal{M}_d(V,W)$. 
    
    \begin{lemma}
        The injection $\psi_d(V,W)\hookrightarrow \mathcal{M}_d(V,W)$ is an equivariant homotopy equivalence.
    \end{lemma}
    \begin{proof}
      We can write down a homotopy inverse $\mathcal{M}_d(V,W) \to \psi_d(V,W)$
      by mapping
      $(M,a)$ to $\varphi_a (M)$, where $\varphi_a$ is the linear scaling of $V$ in the direction of $R$, mapping $a$ to $1$. 
    \end{proof}
    
    \begin{defn}
        For a finite group $H$ and a finite dimensional orthogonal representation $\widetilde{V}$ of $H$, let $\mathcal{N}_d^H(\widetilde{V})$ denote the cobordism set of $H$-manifolds embedded in $\widetilde{V}$, defined as follows. 
        Elements are equivalence classes of closed $d$-dimensional $H$-invariant submanifolds $M\subset D_1(\widetilde{V})$.
        We say two manifolds $M_0, M_1 \subset \widetilde{V}$ are cobordant if there exists a $(d+1)$-dimensional compact $H$-invariant submanifold $N\subset D_1(\widetilde{V})\times [0,1]$ with boundary 
        $$\partial N= \left (M_0\times \{ 0\} \right ) \coprod \left ( M_1\times \{1\} \right )$$ 
        and such that $N$ is cylindrical near $\widetilde{V}\times \{0,1\}$. 
    \end{defn}    
    If $\widetilde{V}$ contains a trivial summand $R$, then $\mathcal{N}_d^H(\widetilde{V})$ becomes a monoid, with composition given as follows.
            If $M_0, M_1 \subset D_1(\widetilde{V})$, the shifted manifolds $M_0' = M_0 + \mathbf{e}_R$ and $M_1' = M_1 - \mathbf{e}_R$ are disjoint, so their union $M = M_0' \cup M_1'$ is a manifold, and is still $H$-invariant, since $R$ is trivial. Rescaling $M$ to be contained in $D_1(\widetilde{V})$ gives a representative for the composite. The usual construction (see \cite[Corollary 3.11]{monoids} for example) of an embedding of $N = M_0 \times [0, 1]$ into $D_1(\widetilde{V}) \times [0, 1]$ so that $\partial N \cap \{ 1 \} = \varnothing$ shows that $\mathcal{N}_d^H(\widetilde{V})$ is in fact a group in this case.

    \begin{lemma}\label{cobgroup}
         For any $H\leq G$, there is an isomorphism of monoids 
         \[ \pi_0 \big (\mathcal{M}_d(V,W)^H \big) \cong \mathcal{N}_{d-\dim(W)}^H(V-W). \]
         In particular the monoid $\pi_0 \big (\mathcal{M}_d(V,W)^H \big)$ is a group.
    \end{lemma}

    Recall
    that being a group is a property of monoids (not extra data): firstly it must have left and right units, which are unique and equal when they both exist, and secondly any element must have a left and right inverse, which are also unique and equal when they exist.
    
    \begin{proof}
      (See also \cite[Proposition 3.6]{monoids}.)  We may replace the left hand side by $\pi_0 \big (\psi_d(V,W)^H \big)$ and consider the map
            \begin{equation}
        \label{eq:1}
        \begin{aligned}
          \mathcal{N}_{d-\dim(W)}^H(V-W) & \to \pi_0 \big (\psi_d(V,W)^H \big)\\
          [M] & \mapsto [M\times W].
        \end{aligned}
      \end{equation}
      To see this is indeed well-defined, first choose a one dimensional subspace $\Span (w)$ of $W$ (recall that we are assuming $W=W^G$ in this subsection). If $N\subset (V-W)\times [0,1]$ is a cobordism between $M_1$ and $M_2$, let 
        \begin{align*}
            N_t=M_1 \times (-\infty, t] \times (W-\Span(w))\: \cup \\
        (N+tw)\times (W-\Span(w))\: \cup \\
        M_2\times [1+t,\infty) \times (W-\Span(w)).
        \end{align*}
        Then $t\mapsto N_t$ gives a path $[-\infty, \infty]\to \psi_d(V,W)^{H}$ from $M_1\times W$ to $M_2\times W$.

        The map~\eqref{eq:1} is a homomorphism: in both cases the monoid structure comes from ``disjoint union'', using the direction of $\mathbf{e}_R$ to ensure disjointness. 
        
        The map~\eqref{eq:1} is surjective: if $L\in \psi_d(V,W)^{H}$, by Sard's theorem we can choose $x\in W$, a regular value of the projection $\pi_W \colon  L\to W$, and consider $M=\pi_W^{-1} (x)$. 
        Then the image of $M$ under~\eqref{eq:1} is in the same path component as $L$. To see this, let $T_t\colon W\to W$ be given by $w\mapsto w+(w-x)t$ (this is equivariant since $W$ is trivial), then 
        $$L_t=(T_t\oplus \textup{id}_{V-W})(L)$$ gives a path $[1,\infty]\to \psi_d(V,W)^{H}$ from $L$ to $M\times W$.
        
        To show injectivity, consider a path $p\colon[0,1]\to \psi_d(V,W)^{H}$ with endpoints $M_1\times W$ and $M_2\times W$. Up to homotopy we can assume by \cref{smoothapprox}, that
        the graph $\Gamma_f \subset [0,1]\times V$ is a smooth $H$-invariant manifold. Then taking the preimage of a regular value $w$ of the projection $\Gamma_f\to W$, gives a cobordism between $M_1$ and $M_2$. 
    \end{proof}

    We would like to conclude that the monoids $\mathcal{M}_d(V,W)^H$ are \emph{group-like} for all $H \leq G$.  Since these topological monoids are non-unital, the notion of being group-like is best defined as left and right multiplication by any element being a weak equivalence.  As in \cref{rmrk:non-unital}, the non-unital monoids $\mathcal{M}_d(V,W)^H$ have \emph{weak} units, namely the elements $(\emptyset,a)$ with $a \in \mathbb{R}_{>0}$: multiplication from the left or from the right by these elements is canonically homotopic to the identity map.  Together with \cref{cobgroup} this implies that the maps $\mathcal{M}_d(V,W)^H \to \mathcal{M}_d(V,W)^H$ defined by left or right multiplication by any element are indeed all homotopy equivalences.
    
    \begin{coro}
      The map
      \begin{align}\label{omegab}
        \mathcal{M}_d(V,W)\to \Omega B \mathcal{M}_d(V,W),
      \end{align}
      adjoint to the inclusion of the 1-skeleton $S^1 \wedge \mathcal{M}_d(V,W) \to B \mathcal{M}_d(V,W)$ is an equivariant equivalence.
    \end{coro}
    \begin{proof}
      As explained above, the topological monoid $\mathcal{M}_d(V,W)^H$ is grouplike for any $H \leq G$, which implies that the canonical map
        \[
            \mathcal{M}_d(V,W)^H\to \Omega B \mathcal{M}_d(V,W)^H
        \]
        is a weak equivalence. But $\Omega B \mathcal{M}_d(V,W)^H\cong \big (\Omega B \mathcal{M}_d(V,W)\big )^H$, which proves \eqref{omegab} is a weak equivalence on $H$-fixed points for any subgroup $H\leq G$.
    \end{proof}
    \begin{defn}
        Let $\mathcal{P}_d(V,W)$ be the topological poset with object space 
        $$\mathrm{Ob}(\mathcal{P}_d(V,W))\subset \psi_d(V,W+R)\times \mathbf{R}$$
        consisting of pairs $(M,a)$ such that $M\cap \pi_R^{-1} (a)=\varnothing$. We say $(M,a) < (M,a')$ if $M=M'$ and $a < a'$.
    \end{defn}
    
    \begin{defn}
        Define a functor $p\colon \mathcal{P}_d(V,W)\to \mathcal{M}_d(V,W)$ given by mapping any object in $\mathcal{P}_d(V,W)$ to the unique object in $\mathcal{M}_d(V,W)$, and mapping the morphism $(M,a)< (M,a')$ to $(N,b)$ where 
        \[
        N=M\cap \pi_R^{-1}((a,a'))-a\mathbf{e}_R \quad and \quad b=a-a'. 
        \]
    \end{defn}
    
    \begin{lemma}\label{intermediate_poset}
        The functor $p\colon \mathcal{P}_d(V,W)\to \mathcal{M}_d(V,W)$ induces a levelwise equivariant equivalence on the nerves.
    \end{lemma}
    \begin{proof}
      This is entirely similar to \cref{315}, using the map
      $N_p \mathcal{M}_d(V,W) \to \mathcal{P}_d(V,W)$ which sends $((M_1,a_1), \dots, (M_p,a_p))$ to $(M,0,a_1,\dots, a_p)$, where $(M,a) = (M_1,a_1) \circ \dots \circ (M_p,a_p)$ denotes the composition in $\mathcal{C}_d(V)$.  This map is equivariant and identifies $N_p \mathcal{M}_d(V,W)$ with the subspace of $N_p \mathcal{P}_d(V,W)$ consisting of $(M,a_0, \dots, a_p)$ where $a_0 = 0$ and where $M\cap \pi_R^{-1}(\mathbf{R}-(0,a_p))=\varnothing$.  An equivariant deformation retraction       to that subspace may be defined by translating in the $R$-direction until $a_0$ becomes 0, and then achieving $M \subset \pi_R^{-1}((0,a_p))$ by pushing the rest of $M$ off to $\pm \infty$ in the $R$-direction.  We again refer to the proof of \cite[Theorem 3.9]{monoids} for a more formal description of this homotopy.
    \end{proof}

Consider the forgetful map $u\colon B\mathcal{P}_d(V,W) \to \psi_d(V,W+R)$ taking a point
         $$(M,a_0,\ldots, a_p, t_0, \ldots , t_p)\in N_p \mathcal{P}_d(V,W) \times \Delta^p \subset B\mathcal{P}_d(V,W)$$
         to $M\in \psi_d(V,W+R)$. Since $u$ is equivariant and $B\mathcal{P}_d(V,W)$ is $G$-connected, for any subgroup $H\leq G$, taking fixed points we get maps
         $$ \big ( B\mathcal{P}_d(V,W) \big )^H \xrightarrow{u^H} \big ( \psi_d(V,W+R) \big)^H_{\varnothing}. $$
         Here $\big ( \psi_d(V,W+R) \big)^H_{\varnothing}$ denotes the component of the $H$-fixed points containing the empty manifold.
    \begin{lemma}\label{410}
         For any subgroup $H\leq G$ the map
         $$ \big ( B\mathcal{P}_d(V,W) \big )^H \xrightarrow{u^H} \big ( \psi_d(V,W+R) \big)^H_{\varnothing} $$
         is a weak equivalence.
    \end{lemma}
    
    \begin{proof}
        By \cref{cobgroup} we know that if $L\in \big ( \psi_d(V,W+R) \big)^H_{\varnothing}$ then for a regular value $x\in W+R$ of $\pi_{R+W}$ the preimage $M=\pi_{R+W}^{-1} (x)$ is equivariantly null-bordant. This allows us to apply analogous arguments to \cite[Proposition 3.20 and 3.21]{monoids} to prove the claim.
    \end{proof}
    
    \begin{rmrk} 
    In \cref{410} it is important to take a component of the fixed points (rather than fixed points of a component), because the group 
    $$\pi_0 \left ( \left ( ( \psi_d (V,W)_{\varnothing} \right ) ^H\right )$$ 
    is not necessarily trivial. Indeed,
        we can identify 
\[
\pi_0 \left ( \left ( ( \psi_d (V,W)_{\varnothing} \right ) ^H\right ) \cong \mathrm{Ker} \big (\mathcal{N}_{d-|W|}^{H} (V-W)\xrightarrow{\upsilon} \mathcal{N}_{d-|W|} (V-W)\big ), 
\]
where $\upsilon$ is the map forgetting the action. The following is an example when this kernel is non-trivial.
Let $G=\mathbf{Z}/2$, $\sigma$ be the sign representation. For a $G$-representation $V$ let $\mathbf{P}(V)$ be the associated projective space. 
Then the manifold $\mathbf{P}(1\oplus \sigma)\coprod \mathbf{P}(1\oplus 1)$ is null-bordant but not equivariantly so, giving a non-trivial element in the kernel (cf. \cite{sinha}).
\end{rmrk}

    \begin{coro}
    The map 
    \[
        \Omega B\mathcal{P}_d(V,W) \xrightarrow{\Omega u} \Omega \psi_d(V,W+R)
    \]
    is an equivariant equivalence.
    \end{coro}
    \begin{proof}
        We need to show $\Omega u$ is a weak equivalence on $H$-fixed points for all $H\leq G$. This is true, since $\big (\Omega \psi_d(V,W+R) \big )^H=\Omega \big ( \psi_d(V,W+R) \big)^H_{\varnothing} $, and so we can apply \cref{410}.
    \end{proof}
    Thus we get the following sequence of equivariant equivalences
    \[
    \psi_d(V,W)\simeq \mathcal{M}_d(V,W) \simeq \Omega B \mathcal{M}_d(V,W) \simeq \Omega B \mathcal{P}_d(V,W) \simeq \Omega \psi_d(V,W+R),
    \]
    showing that~\eqref{looping} is a weak equivalence when $R$ is a one-dimensional trivial representation.  By induction it follows when $R$ is a trivial representation of any dimension.
    
    \subsection{Non-trivial representations}\label{nontriv}
    The main goal for the rest of this section is to prove the following.
    \begin{proposition}\label{eqdeloop}
        The map $\alpha_R\colon \psi_d(V,W)\to \Omega^R \psi_d (V, W+R)$ in \eqref{looping} is an equivariant weak equivalence for an arbitrary $G$-representation $R$, assuming $\dim(W^G)>d$.
    \end{proposition}
    
    The method of the proof is originally due to Segal (\cite{segal}), later refined by Shimakawa (\cite{shima}) and Blumberg (\cite{blum}). It would be interesting to see if the equivariant loop space machines of \cite{caruso_waner} or \cite{may_merling_osorno} could be applied. For the purposes of this paper, we found it easier to give a direct proof using monoidal bar constructions.
    
    The outline is as follows: first, we describe the scanning map which relates spaces of manifolds to mapping spaces. Then we reduce the statement from the loop space (i.e.\ based maps from $S^R$) to a statement about unbased maps from the unit sphere $S(R)$. This allows us to argue locally and finish our proof by an inductive statement using an equivariant triangulation of $S(R)$. Throughout the rest of this section we repeatedly use the idea of computing homotopy fibers via identifying spaces with various bar constructions.

    Let $\pi_R \colon V\to R$ denote orthogonal projection. Let $D_{r,c}(V)\subset V$ denote the open disk in the orthogonal $G$-representation $V$, of radius $r>0$, and centered at $c\in R$ (or at the origin if $c$ is omitted).

\newcommand{\PsiRd}{\mathcal{F}}
\newcommand{\psiRd}{\mathcal{F}_c}
\newcommand{\FF}{\mathcal{F}'}

    \begin{defn}
      For a $G$-invariant open subset $O\subset R$, let
      $\PsiRd(O)$ denote the subspace of $\Psi_d(\pi_R^{-1} (O))$ consisting of those $M$ such that
      $$M\subset D_1 \big ( (W+R)^\bot \big ) \times (W+R). $$
  \end{defn}
  This notation $O \mapsto \PsiRd(O)$ will only be used in subsections~\ref{nontriv}--\ref{scanning_map_section}, throughout which $V$,  $W$, and $R$ are fixed, so we omit them from the notation.
  Let us also mention that invariant open subsets $O \subset R$ are in bijection with open subsets of the quotient space $R/G$, by sending $O \mapsto O/G \subset R/G$, and therefore $\PsiRd$ may be regarded as a presheaf of topological spaces on $R/G$.  It is not hard to see that this presheaf is in fact a sheaf of topological spaces.
  
    For the rest of this section, fix $\epsilon = 1/2$. (In fact any $\epsilon \in (0,1)$ works for the proof.  When talking about scanning maps we often think of $\epsilon$ being small, hence the notation.)
    
    \begin{defn}
        For a subset $C\subset R$ let $C_{\epsilon} \subset R$ denote the open $\epsilon$-neighbor\-hood of $C$ 
        \[
        C_{\epsilon} = \bigcup\limits_{c\in C} D_{\epsilon, c}(R). 
        \]
        Define the scanning map
        \begin{align}\label{scanning}
             \PsiRd(C_\epsilon) \to \Map \big (C, \PsiRd(D_\epsilon(R))\big)
        \end{align}
        as adjoint to the map 
        \begin{equation*}
          \begin{aligned}
            C \times \PsiRd(C_\epsilon) & \to \PsiRd(D_\epsilon(R))\\
            (c,M) & \mapsto (t_c^{-1}(M)) \cap \pi_R^{-1}(D_\epsilon(R)),
          \end{aligned}
        \end{equation*}
        where $t_c \colon R\to R$ is given by $x \mapsto x + c$.
        
        If $C$ is a $G$-invariant subset, then \eqref{scanning} is an equivariant map.
    \end{defn}
    
    For a subset $C\subset R$ we will write $\overline{C}$ to denote the closure of $C$ in $R$.
    Let $S(R)_{\epsilon}$ be the open $\epsilon$-neighborhood of the unit sphere $S(R)$ in $R$, i.e.\ $S(R)_\epsilon = \{ v\in R \mid 1-\epsilon < |v| < 1 + \epsilon \}$. We have the following commutative diagram, whose horizontal maps are scanning maps and vertical maps are restrictions:
    
\begin{equation}\label{segalannulus}
 \begin{tikzcd}
        \PsiRd (\overline{D_{1}(R)}_\epsilon) \arrow[d] \arrow[r] & \Map \big (\overline{D_1(R)}, \PsiRd(D_\epsilon(R)) \big) \arrow[d] \\
        \PsiRd (S(R)_{\epsilon}) \arrow[r]  & \Map \big (S(R), \PsiRd(D_\epsilon(R)) \big)        
 \end{tikzcd}
\end{equation}

Here the top map is clearly an equivalence. Our goal for the rest of this section is to prove in \cref{scanning_map_section} that the bottom map is also an equivariant equivalence when $\dim(W^G) > d$, after discussing some prerequisites in \cref{bar_construction_section} and \cref{homotopy_sheaf_section}. Finally, we show in \cref{finishing_section} that the induced map between the homotopy fibers of the vertical maps in \eqref{segalannulus} can be identified up to equivariant homotopy equivalence with $\alpha_R\colon \psi_d(V,W)\to \Omega^R \psi_d (V, W+R)$, which is therefore also an equivariant equivalence. This will conclude the proof of \cref{eqdeloop} and thus our main theorem.

\subsection{Equivariant bar constructions}\label{bar_construction_section}

In order to prove the statements above, we use bar construction models for certain spaces of manifolds. For these to be useful, we need to discuss some general properties of topological monoids with group actions.

\begin{defn}
  A \emph{topological $G$-monoid} is a topological monoid $\mathcal{M}$ with a left $G$-action, satisfying the compatibility relations
      $$(gm_1)(gm_2)=g(m_1m_2)$$
    for any $g\in G$ and $m_1, m_2 \in \mathcal{M}$.
    
    By a \emph{$G$-space} we shall mean a topological space with a continuous left $G$-action. We consider actions of topological $G$-monoids on $G$-spaces (note that in the following the group $G$ always acts on the left, but the monoid $\mathcal{M}$ can act on either side).
    If $\mathcal{M}$ acts on the $G$-space $Y$ from the left, we say the action is equivariant if 
    $$g(my) = (gm)(gy)$$
    for any $g\in G$, $m\in \mathcal{M}$ and $y\in Y$.
    
    Similarly if $\mathcal{M}$ acts on the $G$-space $X$ from the right, we say the action is equivariant if 
    $$g(xm) = (gx)(gm)$$
    for any $g\in G$, $m\in \mathcal{M}$ and $x\in X$.
\end{defn}

If $\mathcal{M}$ is a topological $G$-monoid, then the fixed point space $\mathcal{M}^H$ is a topological monoid for any $H\leq G$.

\begin{defn}
    We say the topological $G$-monoid $\mathcal{M}$ is \emph{grouplike} if $\mathcal{M}^H$ is grouplike for all $H\leq G$, that is, left and right multiplication by any $x \in \mathcal{M}^H$ defines a weak equivalence $\mathcal{M}^H \to \mathcal{M}^H$.
\end{defn}

\begin{defn}
  Let $\mathcal{M}$ be a topological $G$-monoid, let $X$ and $Y$ be $G$-spaces with $\mathcal{M}$ acting equivariantly from the right on $X$ and from the left on $Y$. Define the two-sided bar construction $B(X,\mathcal{M}, Y)$ as the
  geometric realization of the semi-simplicial $G$-space $N_p(X, \mathcal{M}, Y)= X\times \mathcal{M}^p \times Y$ with the usual face and degeneracy maps.
\end{defn}
Note that $\mathcal{M}$ being a $G$-monoid and the actions being equivariant guarantee that all the face 
maps are equivariant, so $B(X,\mathcal{M}, Y)$ has a natural $G$-action.

We can identify the fixed points $B(X,\mathcal{M}, Y)^H=B(X^H, \mathcal{M}^H, Y^H)$ for any $H\leq G$, since geometric realization preserves equalizers (see \cite[Corollary 11.6]{may_iterated} or \cite[Proposition C.1]{ScholzeNikolaus} where this is proved for thin realization of simplicial spaces, the case of thick realization of semi-simplicial spaces follows by formally adding degeneracies).

\begin{lemma}\label{hocart}
    If $\mathcal{M}$ is a grouplike topological $G$-monoid then for any $Y$ the homotopy fiber of the map $p\colon B(X, \mathcal{M}, Y)\to B(\ast, \mathcal{M}, Y)$ induced by $X \to \ast$ is equivariantly equivalent to $X$. This also implies that the square
    \[
        \begin{tikzcd}
            B(X, \mathcal{M}, Y) \arrow[d] \arrow[r] & B(X, \mathcal{M}, \ast) \arrow[d] \\
            B(\ast, \mathcal{M}, Y) \arrow[r] & B(\ast, \mathcal{M}, \ast)
        \end{tikzcd}
    \]
    is $G$-homotopy cartesian.
  \end{lemma}
  That the square is $G$-homotopy cartesian means that the square induced by taking fixed points for $H \leq G$ is homotopy cartesian, for any $H \leq G$.  Equivalently, the map from the upper left corner to the homotopy pullback of the rest of the diagram is an equivariant weak equivalence, in the path space model for homotopy pullback.
  \begin{proof}
    We
        again use that geometric realization commutes with taking fixed points, and hence the homotopy fiber of the top horizontal map after taking $H$-fixed points is identified with the homotopy fiber of
    $$B(X^H, \mathcal{M}^H, Y^H)\xrightarrow{p^H} B(\ast, \mathcal{M}^H, Y^H).$$
    Since $\mathcal{M}^H$ is grouplike, the map $X^H \to \hofib(p^H)$ is a weak equivalence (by \cite[Proposition 1.6]{segal_cat}).  Similar arguments apply to the bottom horizontal map, and the induced map of horizontal homotopy fibers is weakly equivalent to the identity map of $X^H$.
\end{proof}

\subsection{Homotopy sheaves}\label{homotopy_sheaf_section}

In our study of~\eqref{segalannulus} it will be convenient to use the following special case of a ``homotopy sheaf'' property of $O \mapsto \PsiRd(O)$.

\begin{proposition}\label{homotopy_sheaf}
  Assume $\dim(W^G) > d$ and let $O_1$ and $O_2$ be $G$-invariant open subsets of $R$. Assume there exists an equivariant diffeomorphism $(\delta, f)\colon O_1\cap O_2 \to \mathbf{R} \times Q$, where $\mathbf{R}$ has trivial action and $Q$ is a smooth $G$-manifold such that the identity
    map of $Q$ is equivariantly isotopic to an embedding $Q \to Q$ whose image has compact closure.  Assume further that $\delta$ extends to a continuous map $\delta \colon O_1\cup O_2 \to [-\infty, +\infty]$ such that $O_1\setminus O_2 = \delta^{-1} (-\infty)$ and $O_2 \setminus O_1 = \delta^{-1}(+\infty)$. Then the square of restrictions
    \[
    \begin{tikzcd}
        \PsiRd(O_1\cup O_2) \ar[r] \ar[d] & \PsiRd(O_1) \ar[d] \\
        \PsiRd(O_2) \ar[r] & \PsiRd(O_1\cap O_2)
    \end{tikzcd}
    \]
    is $G$-homotopy cartesian.
\end{proposition}
The assumption on $Q$ prevents wild behavior at infinity.  It is satisfied if $Q$ admits an equivariant proper Morse function with finitely many critical points, or if $Q$ is the interior of an equivariant smooth compact manifold with boundary.

The proof of \cref{homotopy_sheaf} can be summarized
in the following diagram, in which the rightmost square is the one to be shown $G$-homotopy cartesian.
\begin{equation}\label{zigzagsquare}
\begin{tikzcd}[row sep={40,between origins}, column sep={40,between origins}]
      & B(X, \mathcal{M}, Y) \ar{dd}\ar{dl} & & B\mathcal{P}_{X, Y} \ar{ll} \ar{rr} \ar{dd} \ar{dl} & & \PsiRd (O_1 \cup O_2)  \ar{dd} \ar{dl} \\
    B(X, \mathcal{M}, \ast)  \ar{dd} & & B\mathcal{P}_{X} \ar[ll, crossing over] \ar[rr, crossing over]  & & \PsiRd (O_1) \\
      &  B(\ast, \mathcal{M}, Y)  \ar{dl} & & B\mathcal{P}_{Y} \ar{ll} \ar{rr} \ar{dl} & & \PsiRd (O_2) \ar{dl} \\ 
    B(\ast, \mathcal{M}, \ast) & & B\mathcal{P} \ar{ll} \ar{rr}   \ar[from=uu, crossing over] && \PsiRd (O_1\cap O_2) \ar[from=uu,crossing over]
\end{tikzcd}
\end{equation}
After defining the other entries appearing in the diagram, the proof will consist of showing that \cref{hocart} applies to the leftmost square (which will use $\dim(W^G) > d$) and that all horizontal maps are equivariant equivalences (which will use $\dim(W^G) \geq d$).

\newcommand{\translate}{\tau}  

We begin by defining the objects in the leftmost square of~\eqref{zigzagsquare}.  For $a \in \mathbf{R}$, define the \emph{translation map}
$$\translate_a\colon O_1\cap O_2 \to O_1\cap O_2$$ as $\translate_a(x) =  (\delta, f)^{-1}(\delta(x)+a, f(x))$. Let $\pi_R \colon V \to R$ denote orthogonal projection.
For a subset $I \subset [-\infty,\infty]$ we will write $\pi_R^{-1}\delta^{-1}(I) \subset V$ for $\pi_R^{-1}(A)$ where $A = \delta^{-1}(I) \subset O_1 \cup O_2 \subset V$.

\begin{defn}
  Let $\mathcal{M}$ be the following topological $G$-monoid.
    Elements are pairs $(M,a)\in \PsiRd(O_1\cap O_2)\times \mathbf{R}_{> 0}$, satisfying $M\subset \pi_R^{-1}\delta^{-1}(0, a)$. Composition of $(M_1, a_1)$ and $(M_2, a_2)$ is given by $(M,a)$ where 
        \begin{gather*}
    M=M_1 \cup \translate _{a_1} ( M_2 ) \\
    a= a_1 + a_2.
    \end{gather*}
\end{defn}
    
\begin{defn}
  Let $X$ be the space of pairs $(M,a)\in \PsiRd(O_1)\times \mathbf{R}_{> 0}$ such that $M\subset \pi_R^{-1}\delta^{-1}[-\infty, a)$ and
  let $Y$ be the space of pairs $(M, a)\in \PsiRd(O_2)\times \mathbf{R}_{> 0}$ such that $M\subset \pi_R^{-1}\delta^{-1}(-a, +\infty]$. 
    Then the monoid $\mathcal{M}$ acts on $X$ equivariantly from the right, in the following way. If $x=(M_1,a_1)\in X$ and $m=(M_2,a_2)\in \mathcal{M}$ then $xm=(M,a)$ where 
\begin{gather*}
    M=M_1 \cup \translate_{a_1} (M_2) \\
    a = a_1 + a_2.
\end{gather*}
Similarly, $\mathcal{M}$ acts on $Y$ equivariantly from the left the following way. If $m=(M_1,a_1)\in \mathcal{M}$ and $y=(M_2,a_2)\in Y$ then $my=(M,a)$ where 
\begin{gather*}
    M= \translate_{-a}(M_1) \cup M_2 \\
    a = a_1 + a_2.
\end{gather*}
\end{defn}

\begin{lemma}\label{lemma-where-high-dim-is-used}
  Let
  $\mathcal{M}$ be as above and assume $\dim (W^G) > d$. Then $\mathcal{M}^H$ is connected for each $H\leq G$, and in fact
  group-like.  Therefore the leftmost square in~\eqref{zigzagsquare} is $G$-homotopy cartesian.
\end{lemma}
\begin{proof}
    Let $m = (M,a) \in \mathcal{M}$. Consider the projection $\pi_{W^G}\colon M \to W^G$. By Sard's theorem, there exists a regular value $a\in W^G$. Since we are assuming $\dim(W^G) > d$, this means $a$ is not in the image $\pi_{W^G}(M)$. Consider the affine map $L_t\colon W^G \to W^G$ given by $L_t(x) = (1-t)(x-a) + a$, and let $\varphi_t = id_{V-W^G} \oplus L_t \colon V \to V$. Then $m_t = (\varphi_t^{-1} (M), a)$ gives a path from $m$ to $(\varnothing, a)$ which is in the path component of $e$. If $m\in \mathcal{M}^H$ then the path $m_t$ constructed above is in fact a path in $\mathcal{M}^H$, showing that $\mathcal{M}^H$ is path-connected.  Moreover, left or right multiplication by $m$ is now homotopic to multiplication by $(\varnothing,a)$, which is homotopic to the identity.  This shows that $\mathcal{M}^H$ is group-like for any $H \leq G$, so \cref{hocart} applies to the leftmost square in~\eqref{zigzagsquare}.
  \end{proof}

Next we define the entries appearing in the middle square of~\eqref{zigzagsquare}.

\begin{defn}
  Let $\mathcal{P}$ be the following topological
   poset. Objects are pairs $(M, a) \in \PsiRd(O_1\cap O_2) \times \mathbf{R}$ satisfying
    $$
    M \cap \pi_R^{-1}\delta^{-1}(a) = \varnothing.
    $$
    We say
    $(M,a) < (M', a')$ when $M=M'$ and $a < a'$.
    
    Similarly define $\mathcal{P}_X$, $\mathcal{P}_Y$, $\mathcal{P}_{X, Y}$ as posets of pairs $(M,a)$, where $M$ is in $\PsiRd (O_1)$, $\PsiRd(O_2)$ and $\PsiRd(O_1\cup O_2)$ respectively.
\end{defn}

We have restriction functors
\[
\begin{tikzcd}
    \mathcal{P}_{X, Y} \ar[r]\ar[d] & \mathcal{P}_X \ar[d] \\
    \mathcal{P}_Y \ar[r] & \mathcal{P}.
\end{tikzcd}
\]
We'll use the classifying spaces of these posets to approximate between the bar constructions and spaces of manifolds.

\begin{defn}
    Let $\mathcal{P} \to \mathcal{M}$ be the functor given as follows. Every object of $\mathcal{P}$ maps to the unique object in $\mathcal{M}$. A morphism $(M, a_0) < (M, a_1)$ in $\mathcal{P}$ maps to $(\widetilde{M}, \widetilde{a}) \in \mathcal{M}$, where 
    $$\widetilde{a} = a_1-a_0 $$
    $$\widetilde{M} = \translate_{-a_0} M \cap \pi_R^{-1}\delta^{-1}(a_0, a_1). $$
\end{defn}

\begin{lemma}\label{poset_monoid_levelwise}
    The functor $\mathcal{P} \to \mathcal{M}$ induces a levelwise equivariant equivalence on the nerves.
\end{lemma}

\begin{proof}
    A one-sided levelwise inverse $N_p \mathcal{M} \to N_p \mathcal{P}$ is given by 
    \[
        (M_1, \ldots, M_p, a_1, \ldots, a_p) \mapsto (M, 0, a_1, \ldots, a_p),
    \]
    where $M$ is defined by $(M, a) = (M_1, a_1) \cdots (M_p, a_p) \in \mathcal{M}$.

    The composition $N_p \mathcal{M} \to N_p \mathcal{P} \to N_p \mathcal{M}$ is equal to the identity map, and the composition $N_p \mathcal{P} \to N_p \mathcal{M} \to N_p \mathcal{P}$ is homotopic to the identity map.  The homotopy is again given by sliding, as in \cref{315} and \cref{intermediate_poset}, this time in the $\R$ direction under
        the decomposition $O_1 \cap O_2 \cong Q \times \R$.  Since this diffeomorphism is assumed equivariant for the trivial action on $\R$, we again produce an equivariant homotopy.
\end{proof}

\begin{defn}\label{poset_to_bar}
Let $N_\bullet \mathcal{P}_X \to N_\bullet (X, \mathcal{M}, \ast)$ be the simplicial map given by 
$$(M, a_0, \ldots, a_p) \mapsto (x, a_0, \widetilde{M}_1, \widetilde{a}_1 \ldots, \widetilde{M}_p, \widetilde{a}_p), $$
where 
$$\widetilde{a}_i = a_i - a_{i-1}$$ 
$$ x = \pi_R^{-1}\delta^{-1} [-\infty, a_0) \cap M $$
$$\widetilde{M}_i = \translate_{-a_{i-1}}M \cap \pi_R^{-1}\delta^{-1}(0, \widetilde{a}_i). $$
\end{defn}

\begin{lemma}
    The map $N_\bullet \mathcal{P}_X \to N_\bullet (X, \mathcal{M}, \ast)$ is a levelwise equivariant equivalence.
\end{lemma}
\begin{proof}
    The proof is completely analogous to \cref{poset_monoid_levelwise}.
\end{proof}

Analogously to \cref{poset_to_bar} we can write down simplicial maps between the nerves
\begin{align*}
    N_\bullet \mathcal{P}_Y & \to N_\bullet(\ast, \mathcal{M}, Y) \\
    N_\bullet \mathcal{P}_{X, Y} & \to N_\bullet(X, \mathcal{M}, Y),
\end{align*}
which are levelwise equivariant homotopy equivalences. The maps are given by taking the slices of $M \in N_\bullet \mathcal{P}$ and shifting them to be elements of $\mathcal{M}$. The levelwise inverses are given by inclusions as in \cref{poset_monoid_levelwise}.

\begin{defn}\label{poset_to_psi_forgetful}
    Define the forgetful maps 
    \begin{align*}
    B\mathcal{P} & \to \PsiRd (O_1\cap O_2) \\
    B\mathcal{P}_X & \to \PsiRd (O_1) \\
    B\mathcal{P}_Y & \to \PsiRd (O_2) \\
    B\mathcal{P}_{X, Y} & \to \PsiRd (O_1\cup O_2)
    \end{align*}
    by 
    $$ (M, a_0, \ldots, a_p) \mapsto M. $$
\end{defn}

These 
maps are always defined, but need not be equivariant equivalences unless $\dim(W^G) \geq d$.  For example, it follows from Lemma~\ref{poset_monoid_levelwise} that $B\mathcal{P}^H$ is path connected for all $H < G$, which $(\PsiRd(O_1 \cap O_2))^H$ need not be.

\begin{lemma} \label{forgetful_lemma}
  The forgetful maps in \cref{poset_to_psi_forgetful} are equivariant weak equivalences, provided $\dim(W^G) \geq d$.
\end{lemma}

\begin{proof}
  The 
  proof is similar in all cases, let us outline the case of $B\mathcal{P}$.  By abuse of notation we identify $V = R \times W \times (R + W)^\perp$.  A point in $\PsiRd(O_1 \cap O_2)$ is then a submanifold $M \subset (O_1 \cap O_2) \times W \times (R + W)^\perp$ which is closed as a subset, and satisfies $M \subset (O_1 \cap O_2) \times W \times D_1((R + W)^\perp)$.  Let us first give a careful construction of a path from an arbitrary $M \in \PsiRd(O_1 \cap O_2)$ to a point in the image from $B\mathcal{P}$.

  We have projections
  \begin{equation*}
    \begin{aligned}
      M \xrightarrow{\pi_{W^G} \vert_{M}} W^G\\
      M \xrightarrow{\pi_R} O_1 \cap O_2 \xrightarrow{\delta} \R
    \end{aligned}
  \end{equation*}
  and if $M \in (\PsiRd(O_1 \cap O_2))^H$ for some $H < G$ so that $M$ inherits an action by $H$, then both of these functions are invariant under the action of $H$.  By Sard's theorem there are plenty of regular values for
  \begin{equation}
    \label{eq:2}
    (\pi_{W^G}\vert_{M},\delta \circ \pi_R): M \to W^G \times \R
  \end{equation}
  and by the assumption $\dim(W^G) \geq d = \dim(M)$ these regular values are not in the image.  If $(w_0,a_0) \in W^G \times \R$ is not in the image, the strategy is now to amplify from the image being ``disjoint from $\{(w_0,a_0)\} \subset W^G \times \R$'' to the image being ``disjoint from $W^G \times \{a_0\} \subset W^G \times \R$'', by pushing the parts of $M$ that intersect $\delta^{-1}\pi_R^{-1}(a_0)$ to infinity in the $W^G$-direction, radially away from $w_0$.

  For technical reasons it will be a slight problem
    that~\eqref{eq:2} is not proper: firstly the set of regular values is not necessarily open in $W^G \times \R$ which makes the ``pushing to infinity'' harder to implement, and secondly the condition that $(w,a) \in W^G \times \R$ is regular is not an open condition on $M$ (i.e., the set of $M$'s satisfying it is not an open subset of $\PsiRd(O_1 \cap O_2)$).  Therefore we choose
    a compact subset $K \subset Q$ such that the identity map of $Q$ is equivariantly isotopic to an embedding with image in $K$, and restrict attention to the closed subset
  \begin{equation*}
    M' = M \cap \pi_{W-W^G}^{-1}(D_1(W-W^G)) \cap \pi_R^{-1}(f^{-1}(K)) \subset M.
  \end{equation*}
  The restriction
  \begin{equation*}
    (\pi_{W^G}\vert_{M},\delta \circ \pi_R)\vert_{M'}: M' \to W^G \times \R
  \end{equation*}
  is then a proper map, and the image is a closed subset of $W^G \times \R$.  Sard's theorem now shows that the complement of the image is open and dense, and we may choose $(w_0,a_0) \in W^G \times \R$ and $\epsilon > 0$ such that
  \begin{equation*}
    M' \cap \pi_{W^G}^{-1}(D_\epsilon(w) \times (a-\epsilon,a+\epsilon)) = \emptyset.
  \end{equation*}
  Now choose an equivariant isotopy $t \mapsto \rho_t$ from the identity map of $Q$ to an embedding $Q \to K \subset Q$, another equivariant isotopy $t \mapsto \rho_t'$ from the identity map of $W - W^G$ to an embedding $W - W^G \to D_1(W - W^G)$, and a third equivariant isotopy $t \mapsto \rho_t''$ from the identity map of $W^G$ to an embedding $W^G \to D_{\epsilon,w_0}(W^G)$.  Finally choose $\lambda: \R \to \R$ with $[a_0-\frac\epsilon2,a_0 + \frac\epsilon2] \subset \lambda^{-1}(1)$ and $\mathrm{supp}(\lambda) \subset (a_0-\epsilon,a_0 + \epsilon)$, and assemble these isotopies to an isotopy of equivariant self-embeddings
  \begin{equation*}
    \begin{aligned}
      (\R \times Q) \times W^G \times (W - W^G) \times (R + W)^\perp & \xrightarrow{\rho'''_t} (\R \times Q) \times W \times (R + W)^\perp\\
      ((a,q),w,w',z) & \mapsto ((a,\rho_{t\lambda(a)}(q)),\rho''_{t\lambda(a)}(w),\rho'_{t \lambda(a)}(w'),z),
    \end{aligned}
  \end{equation*}
  where $t \in [0,1]$.  For $t = 0$ it is the identity, and for $t = 1$ it has the property that $\pi_R^{-1}\delta^{-1}([a_0-\frac\epsilon2,a_0+\frac\epsilon2])$ is mapped into the subset
  \begin{equation*}
    ([a_0 - \tfrac\epsilon2,a_0 + \tfrac\epsilon2]\times K) \times D_{\epsilon,w_0}(W^G) \times D_1(W - W^G) \times (R + W)^\perp,
  \end{equation*}
  which is disjoint from $M$.  Moreover it restricts to the constant isotopy outside $\pi_R^{-1}\delta^{-1}([a_0 - \epsilon,a_0 + \epsilon])$.  The isotopy $t \mapsto \rho'''_t$ then gives a path
  \begin{align*}
    [0,1] \to \PsiRd(O_1 \cap O_2)\\
    t \mapsto (\rho'''_t)^{-1}(M),
  \end{align*}
  from $M$ to a manifold $(\rho'''_1)^{-1}(M)$ which is disjoint from $\pi_R^{-1}\delta^{-1}([a_0-\frac\epsilon2,a_0+\frac\epsilon2])$.  In particular $((\rho'''_1)^{-1}(M),a_0)$ is an object of $\mathcal{P}$.

  This shows that the forgetful map $B\mathcal{P} \to \PsiRd(O_1 \cap O_2)$ induces a surjection on $\pi_0$.  We will use a parametrized version of this argument to show that it is in fact a weak equivalence.  Consider a lifting problem
  \begin{equation*}
    \begin{tikzcd}
      \partial D^q \arrow[r] \arrow[d] & B\mathcal{P}^H \arrow[d] \\
      D^q \arrow[r, "f"] \arrow[ru, dashed] & \PsiRd (O_1\cap O_2)^H
    \end{tikzcd}
  \end{equation*}
  with $q \in \mathbb{N}$, in which we wish to find the diagonal arrow after possibly changing the horizontal arrows by  homotopies through commutative squares.  Let us write $f(x) = M_x \subset (O_1 \cap O_2) \times W \times (W+R)^\perp$.  By openness of the conditions in the above argument, and by compactness of $D^q$, we may then choose finitely many $(w_1,a_1), \dots, (w_m,a_m) \in W^G \times \R$ and corresponding  $\epsilon_1, \dots, \epsilon_m > 0$ and open subsets $U_1, \dots, U_m \subset D^q$ such that
  \begin{equation*}
    M_x \cap \pi_{W-W^G}^{-1}(D_1(W-W^G)) \cap \pi_R^{-1}(f^{-1}(K)) \cap \pi_{W^G}^{-1}(D_{\epsilon_i}(w_i) \times (a_i-\epsilon_i,a+\epsilon_i)) = \emptyset
  \end{equation*}
  whenever $x \in U_i$.  After possibly shrinking the $\epsilon_i$ and perturbing the $a_i$, we may arrange that all $\epsilon_i$ equal some $\epsilon$, and that the intervals $[a_i - \epsilon,a_i + \epsilon]$ are disjoint.  Letting $t \mapsto \rho'''_t = \rho^i_t$ be the isotopy of self-embeddings of $(O_1 \cap O_2) \times W \times (R+W)^\perp$ constructed as above with $w_i$ and $a_i$ in place of $w_0$ and $a_0$ gives a homotopy
  \begin{align*}
    [0,1] \times U_i & \to \PsiRd(O_1 \cap O_2)\\
    (t,x) & \mapsto (\rho^i_t)^{-1}(M_x)
  \end{align*}
  from the identity to a map lifting to $U_i \to N_0 \mathcal{P} \subset B\mathcal{P}$.  If we instead use $\rho^i_{t \lambda_i(x)}$ for a bump function $\lambda_i: D^q \to [0,1]$ then the homotopy extends by the constant homotopy outside $U_i$, and since the $\rho^i_t$ commute for different $i = 1, \dots, m$ we may glue these together to a homotopy of maps $[0,1] \times D^q \to \PsiRd(O_1 \cap O_2)$ starting at $f$ and ending at a map lifting to $B\mathcal{P}$.  If we choose the $a_i$ far from any $a \in \R$ appearing in the given map $\partial D^q \to B\mathcal{P}^H$, then there is a compatible homotopy of the entire square after which the diagram admits a diagonal lift.
\end{proof}

\begin{proof}[Proof of \cref{homotopy_sheaf}]
  We
  have now constructed the entire diagram~\eqref{zigzagsquare}, which is a commutative diagram in the category of $G$-spaces, and shown that all horizontal maps are equivariant weak equivalences.  Lemma~\ref{lemma-where-high-dim-is-used} implies that the topological $G$-monoid $\mathcal{M}$ is grouplike, so by \cref{hocart} the leftmost square is $G$-homotopy cartesian and hence the rightmost one is too.
\end{proof}

\subsection{The scanning map is an equivalence}\label{scanning_map_section}

In
this section we use the results of the previous two subsections to prove the following, again following the idea from \cite{segal}.

\begin{proposition}\label{bottommap}
    The scanning map $\PsiRd (S(R)_{\epsilon}) \to \Map \big (S(R), \PsiRd(D_\epsilon(R)) \big)$ in \eqref{segalannulus} is an equivariant equivalence, provided $\dim(W^G) > d$.
\end{proposition}

Before explaining the strategy, let us make the following definition, inspired by \cite[Section 4.1]{chris_sch}.

\begin{defn}
  For a $G$-invariant open subset $O \subset S(R)$, a \emph{scanning function} for $O$ is a smooth $G$-invariant function $\rho: S(R) \to [0,1]$ such that $O = S(R) \setminus \rho^{-1}(0)$, and such that for all $x \in O$ the ball in $R$ of radius $\rho(x)$ is disjoint from $S(R) \setminus O$.

  Let $p_\epsilon \colon S(R)_\epsilon \to S(R)$ denote radial projection $x \mapsto \frac{x}{|x|}$, and for $G$-invariant open subsets $O \subset S(R)$ let
   $\FF(O) \subset \PsiRd(p_\epsilon^{-1}(O)) \times C^\infty(S(R),[0,1])$ be the subspace consisting of pairs $(M,\rho)$ where $\rho$ is a scanning function for $O$.

  We then define an equivariant map
  \begin{equation}
    \label{eq:3}
    \FF(O) \to \Map \big(O,\PsiRd(D_\epsilon(R))\big)
  \end{equation}
  for any $G$-invariant open subset $O \subset S(R)$, as adjoint to
  \begin{equation*}
    \begin{aligned}
      O \times \FF(O) & \to \PsiRd(D_\epsilon(R))\\
      (c,(M,\rho)) & \mapsto t_{c,\rho}^{-1}(M),
    \end{aligned}
  \end{equation*}
  where $t_{c,\rho}: R \to R$ is given by $t_{c,\rho}(x) = \rho(c) \cdot (x + c)$.
\end{defn}

Clearly the forgetful map $\FF(O) \to \PsiRd(p_\epsilon^{-1}(O))$ is an equivariant weak equivalence (the space of scanning functions for $O$ is a convex subspace of a vector space; non-emptiness may be seen by averaging a non-invariant bump function).  For $O = S(R)$ it has the preferred homotopy inverse $M \mapsto (M,1)$.  Under these identifications, clearly the scanning map~(\ref{eq:3}) for $O = S(R)$ is a model for the bottom map in the diagram~(\ref{segalannulus}).

We will use a local-to-global argument to show that the scanning map is an equivariant weak equivalence for many subsets $O \subset S(R)$ including $O = S(R)$, using that both domain and codomain satisfy a homotopy sheaf property.  We begin by proving this for open $G$-invariant $O \subset S(R)$ of a very special form.
\begin{defn}
  A $G$-invariant open subset $O \subset S(R)$ is \emph{elementary} if it is the union of disjoint sets of the form $G.U$, where $U \subset S(R)$ is an open subset contained in some hemisphere, and such that for each $g \in G$ either $gU = U$ or $U \cap gU = \emptyset$.  Let us write $G_U = \{g \in G \mid gU = U\}$.

  We further require that for each $U$ there exists a $G_U$-fixed point $b \in U$ and a $G_U$-equivariant diffeomorphism
  \begin{equation*}
    (R - \R b) \to U.
  \end{equation*}
\end{defn}
If $O = G.U$ for such a $U$ then $G_U = \{g \in G \mid gU = U\}$ is a subgroup of $G$ contained in $G_b$, and the action map induces a $G$-equivariant diffeomorphism $G \times_{G_U} U \xrightarrow{\approx} O$.  For a general elementary $O \subset S(R)$ we have
\begin{equation*}
  \coprod_{i = 1}^k G \times_{H_i} U_i \xrightarrow{\approx} O \subset S(R)
\end{equation*}
for subgroups $H_i \subset G$ acting on $U_i \subset S(R)$.

\begin{lemma}\label{lemma:elementary-subset}
  The scanning map~\eqref{eq:3} is an equivariant weak equivalence when $O \subset S(R)$ is elementary.
\end{lemma}
\begin{proof}
  Both domain and codomain of~(\ref{eq:3}) take disjoint union of $G$-invariant open subsets to product of $G$-spaces, so it suffices to consider the case where $O = G.U$ for a single $U$.   Let us write $H = G_U$ so that our assumptions imply a $G$-equivariant diffeomorphism
  \begin{equation*}
    G \times_H (R - \R x) \xrightarrow{\approx} O.
  \end{equation*}
  It follows that the map
  \begin{equation*}
    \begin{aligned}
      \PsiRd(O) & \to \Map_H(G,\PsiRd(U))\\
      M & \mapsto (g \mapsto U \cap g M)
    \end{aligned}
  \end{equation*}
  is a $G$-equivariant homeomorphism (it is a bijection since a submanifold $M \subset O \times D_1((W+R)^\perp) \times W$ is uniquely determined by its intersection with each of the $G/H$ many copies of $U \times D_1((W+R)^\perp) \times W$).

  Evidently a similar homeomorphism holds for the codomains of the scanning map
  \begin{equation*}
    \Map \big(O,\PsiRd(D_\epsilon(R))\big) \xrightarrow{\approx} \Map_H\bigg(G,\Map \big(U,\PsiRd(D_\epsilon(R))\big)\bigg),
  \end{equation*}
  so it suffices to compare $\PsiRd(U)$ and $\Map(U,\PsiRd(D_\epsilon(R)))$ equivariantly for $H$.  After choosing an $H$-equivariant diffeomorphism $R - \R x \approx U$, a homotopy inverse to the scanning map is defined by evaluating at the origin $0 \in R - \R x$.
\end{proof}

Next we explain what will be an ``induction step'' in the local-to-global argument.
\begin{lemma}\label{lemma:induction}
  Assume $\dim(W^G) > d$ and let $O_1, O_2 \subset S(R)$ satisfy the assumptions of~\cref{homotopy_sheaf}.  If~\eqref{eq:3} is an equivariant weak equivalence for $O = O_1$, $O = O_2$, and $O = O_1 \cap O_2$, then it is also an equivariant weak equivalence for $O = O_1 \cup O_2$.
\end{lemma}
\begin{proof}
  The scanning map~\eqref{eq:3} may be used as the horizontal maps in a diagram
  \begin{equation}
    \label{eq:5}
    \begin{tikzcd}
      \FF(O_1) \arrow[r] \arrow[d] & \FF(O_1 \cap O_2) \arrow[d] & \FF(O_2) \arrow[d] \arrow[l] \\
      \Map \big(O,\PsiRd(D_\epsilon(R))\big) \arrow[r] & \Map \big(O,\PsiRd(D_\epsilon(R))\big) & \Map \big(O,\PsiRd(D_\epsilon(R))\big). \arrow[l]
    \end{tikzcd}
  \end{equation}
  The horizontal maps on the bottom row are the restriction maps, while the maps in the top row are given by
  \begin{equation*}
    (M,a) \mapsto (M \cap O_1 \cap O_2,a \cdot \lambda),
  \end{equation*}
  where $\lambda: S(R) \to [0,1]$ is a scanning function for $O_1 \cap O_2$.
  
  The diagram~\eqref{eq:5} is not quite commutative on the nose, but each square commutes up to a canonical equivariant homotopy.  That suffices to induce an equivariant map of homotopy pullbacks, which will be an equivariant weak equivalence under our assumptions.  Now by \cref{homotopy_sheaf} the canonical map from $\FF(O_1 \cup O_2)$ to the homotopy pullback of the top row is an equivariant weak equivalence, and for elementary reasons the canonical map from $\Map \big(O_1 \cup O_2,\PsiRd(D_\epsilon(R))\big)$ to the homotopy pullback of the bottom row is also an equivariant weak equivalence.  Up to equivariant weak equivalence the induced map of homotopy pullbacks is therefore identified with~\eqref{eq:3} for $O = O_1 \cup O_2$.
\end{proof}

\begin{proof}[Proof of~\ref{bottommap}]  
  The proof is by induction over simplices in an equivariant smooth triangulation of $S(R)$.  This means a simplicial complex $K$ equipped with a $G$-action by simplicial maps, and an equivariant homeomorphism $j: |K| \to S(R)$ such that for any $p \in \N$ and any $p$-simplex $\sigma$ of $K$, the composition $\Delta^p \hookrightarrow |K| \to S(V)$ is a smooth embedding.  For an ordered $p$-simplex $\sigma$ of $K$ we shall use the same notation for the associated map $\sigma: \Delta^p \to |K|$.  We may arrange that the stabilizer of $j\circ \sigma(\Delta^p) \subset S(R)$ as a subset agrees with the pointwise stabilizer (replace $K$ by its barycentric subdivision if not) and denote this stabilizer subgroup by $G_\sigma < G$.  We shall write $b(\sigma) = \sigma(\frac1{p+1}, \dots, \frac1{p+1}) \in S(R)$ for the \emph{barycenter} of $\sigma$.  It comes with canonical neighborhoods  \begin{equation*}
    b(\sigma) \in \mathrm{star}(\sigma) \subset \overline{\mathrm{star}}(\sigma) \subset |K|,
  \end{equation*}
  the open and closed stars of $\sigma$, which are PL homeomorphic to $\R^{r-1}$ and $I^{r-1}$ respectively, where $r = \dim(R)$.  The point $b(\sigma)$ is fixed by $G_\sigma$, and the open and closed stars are invariant under $G_\sigma$.  We may arrange (again by subdividing) that $\mathrm{star}(\sigma) \cap g(\mathrm{star}(\sigma)) = \emptyset$ for $g \not\in G_\sigma$.  Let us write
  \begin{equation*}
    U_\sigma = j(\mathrm{star}(\sigma)) \subset S(R),
  \end{equation*}
  which is an open $G_\sigma$-invariant neighborhood of $j(b(\sigma))$.  It induces up to a $G$-equivariant diffeomorphism
  \begin{equation*}
    G \times_{G_\sigma} U_\sigma \to G.U_\sigma \subset S(R).
  \end{equation*}
  The open subset $U_\sigma \subset S(R)$ is PL homeomorphic to $\R^r \approx T_{j(b(\sigma))} S(R) \approx (R - \R j(b(\sigma)))$, but we shall make the following additional assumption on the triangulation (we do not know whether it is automatic, but we shall explain below that an equivariant smooth triangulation where it holds may be chosen):
  \begin{itemize}
  \item there exists a $G_\sigma$-equivariant diffeomorphism $(R - \R j(b(\sigma))) \to U_\sigma$.
  \end{itemize}
  Under this assumptions the open subset $G.U_\sigma \subset S(R)$ is elementary.

  We shall impose one additional regularity condition on the equivariant triangulation, about the open subset
  \begin{equation*}
    U_\sigma \setminus j\circ \sigma(\Delta^p) = j(\mathrm{star}(\sigma) \setminus \sigma(\Delta^p)) \subset S(R).
  \end{equation*}
  The corresponding subset $\mathrm{star}(\sigma) \setminus \sigma(\Delta^p) \subset |K|$ is the \emph{deleted open star}.  We shall make the following assumption:
  \begin{itemize}
  \item there exists a $G_\sigma$-equivariant diffeomorphism
    \begin{equation}\label{eq:6}
      U_\sigma \setminus j\circ\sigma(\Delta^p) \xrightarrow{(\delta,f)} (0,1) \times (\mathring{\Delta}^p \times S(N))
    \end{equation}
    for some representation $N$, such that the first coordinate extends to a continuous map $\overline{U_\sigma} \setminus j\circ \sigma(\partial \Delta^p) \to [0,1]$ such that $\delta^{-1}(1) = j \circ \sigma(\mathring{\Delta}^p)$ and $\delta^{-1}(0) = \partial \overline{U_\sigma} \setminus j \circ \sigma(\partial \Delta^p)$.
  \end{itemize}
  A slightly weaker statement, with ``equivariant diffeomorphism'' replaced by ``PL homeomorphism'', is in fact automatic.  This may be checked in $|K|$ where we have the usual decomposition $\overline{\mathrm{star}}(\sigma) \approx \sigma(\Delta^p) \ast \mathrm{link}(\sigma)$, which restricts to a $G_\sigma$-equivariant PL homeomorphism
  \begin{equation*}
    \mathrm{star}(\sigma) \setminus \sigma(\Delta^p) \approx (\mathring{\Delta}^p \ast \mathrm{link}(\sigma)) \setminus ((\emptyset \ast \mathrm{link}(\sigma)) \cup (\mathring{\Delta}^p \ast \emptyset)) \cong (0,1) \times \mathring{\Delta}^p \times \mathrm{link}(\sigma),
  \end{equation*}
  and the link of $\sigma$ is PL homeomorphic to the unit sphere of the normal bundle.

  There is an abstract result (see \cite{equi_triangulation}) that any equivariant smooth manifold admit an equivariant smooth triangulation, but in general it seems unclear to us whether it will necessarily satisfy the two extra conditions.  In the case of a unit sphere in an orthogonal representation this question can be circumvented by choosing an equivariant triangulation by \emph{spherical simplices}, for which the conditions may be checked by hand.  See \cref{rmrk:nice-triangulation} below for more details.
  
  Let us choose an equivariant smooth triangulation with the above two regularity properties, and continue with the proof of the Proposition.  For each subcomplex $\Lambda \subset K$ we have the open subset
  \begin{equation*}
    O^\Lambda = S(V) \setminus j(|\Lambda|),
  \end{equation*}
  and we shall consider the statement that the scanning map~\eqref{eq:3} is an equivariant weak equivalence when $O = O^\Lambda$.  This is of course true for $\Lambda = K$ where $O^\Lambda = \emptyset$ and is also true when $\Lambda \subset K$ is the $(d-1)$-skeleton, since $O^\Lambda$ is then the union of (the open stars of) the top dimensional simplices, which is elementary by our discussion above.

  Let us now assume $\Lambda \subset K$ and let $\sigma < \Lambda$ be a simplex which is maximal (not a proper face of any other simplex in $\Lambda$).  Then $\Lambda \setminus \{\sigma\} \subset \Lambda$ is again a simplicial complex, and $\Lambda \setminus G.\sigma$ is an equivariant subcomplex.  We also have subcomplexes
  \begin{align*}
    K_\sigma &= \{ \tau < K \mid \text{$g.\sigma \not \subset \tau$ for any $g \in G$}\}\\
    K'_\sigma & = K_\sigma \cup \Lambda = K_\sigma \cup G\sigma,
  \end{align*}
  for which $O^{K_\sigma} = S(V) \setminus j(|K_\sigma|)$ equals $G.j(\mathrm{star}(\sigma)) \cong G \times_{G_\sigma} U_\sigma$, and hence by the discussion above is elementary.

  Now observe that the intersection
  \begin{equation*}
    O^{K_\sigma} \cap O^\Lambda = S(R) \setminus j(|K_\sigma \cup \Lambda|) = O^{K'_\sigma}
  \end{equation*}
  equals the subset $G \times_{G_\sigma} (U_\sigma \setminus j \circ \sigma(\Delta^p)) \cong G.(U_\sigma \setminus j \circ \sigma(\Delta^p)) \subset S(R)$, while
  \begin{equation*}
    O^{K_\sigma} \cup O^\Lambda = O^{\Lambda \setminus G.\sigma}.
  \end{equation*}
  Since both $\Lambda$ and $K'_\sigma$ have strictly more simplices, we may assume by downwards induction that the scanning map~(\ref{eq:3}) is an equivariant weak equivalence for both $O = O^{\Lambda}$ and $O = O^{K'_\sigma}$, and we have already seen that it is for $O = O^{K_\sigma}$ since that subset is elementary.  The $G_\sigma$-equivariant diffeomorphism~(\ref{eq:6}) induces up to a $G$-equivariant diffeomorphism which shows that Lemma~\ref{lemma:induction} applies for $O_1 = O^\Lambda$ and $O_2 = O^{K'_\sigma}$.  The conclusion of that Lemma is then that the scanning map is also an equivariant weak equivalence for $O = O^{\Lambda \setminus G.\sigma}$ and by induction we conclude it is an equivariant weak equivalence for $O = O^\emptyset = S(R)$.
\end{proof}

\begin{rmrk}\label{rmrk:nice-triangulation}
  As
  promised, let us discuss how to construct an equivariant triangulation of $S(R)$ with the good properties above.  This is presumably well known to experts but we were unable to find a reference so we indicate a construction here.
  
  The spherical geometry is used to construct canonical simplices with a given set of vertices: if $\sigma = (x_0, \dots, x_p) \in S(R)^{p+1}$ is an ordered tuple of unit vectors which are linearly independent in $R$ and are contained in an open hemisphere (i.e.\ there exists an $x \in S(R)$ with $\langle x,x_i\rangle > 0$ for all $i$) then the map
  \begin{equation}\label{eq:9}
    \begin{aligned}
      j_\sigma: \Delta^p & \to S(R)\\
      (t_0, \dots, t_p) & \mapsto \frac{t_0 x_0 + \dots + t_p x_p}{|t_0 x_0 + \dots + t_p x_p|}
    \end{aligned}
  \end{equation}
  is well defined and injective, and its image is the spherical convex hull of the subset $\{x_0, \dots, x_p\} \subset S(R)$.  Let us say that a simplex $j: \Delta^p \to S(R)$ is \emph{spherical} if it is of this form.  As we shall recall shortly, it is not hard to construct an equivariant triangulation in which all simplices are spherical, but let us first discuss some convenient consequences of this additional condition.

  Let $j: \Delta^p \to S(R)$ be a spherical simplex with vertices $x_i = j(e_i)$ where $e_i \in \Delta^p$ is the $i$th vertex.  If $b \in S(R)$ is a point such that $\langle b,x_i\rangle > 0$ for all $i$, then $j$ may be factored as $\Delta^p \to (R - \R b) \to S(R)$, where the second map is the \emph{gnomonic projection}
  \begin{equation}\label{eq:8}
    \begin{aligned}
      R - \R b & \to S(R)\\
      y & \mapsto \frac{b + y}{|b + y|}.
    \end{aligned}
  \end{equation}
  The gnomonic projection sends $0 \mapsto b$ and defines a diffeomorphism from $R - \R b$ onto the open hemisphere $\{y \in S(R) \mid \langle y,b\rangle > 0\}$, so $j$ factors uniquely when its image is contained in that hemisphere.  In general the inverse of the gnomonic projection takes hemispheres in $S(R)$ to affine half-spaces in $R - \R b$, and hence the image $j(\Delta^p)$ is taken to a euclidean simplex in $R - \R b$, which in particular is a closed and convex subset.

  A particularly useful choice is $b = b_\sigma = j(\frac1{p+1},\dots,\frac1{p+1})$, which has the convenient property that it is fixed by any $g \in G_\sigma$, where $G_\sigma < G$ denotes the subgroup  stabilizing $\{x_0, \dots, x_p\} \in 2^{S(V)}$.  For this choice $b = b_\sigma$ the factorization through the gnomonic projection $R - \R b_\sigma \to S(R)$ is equivariant for $G_\sigma$.

  Let us now consider an equivariant triangulation $j: |K| \to S(R)$ in which all simplices are spherical and sufficiently small.  As in any triangulation of a manifold the closed star of a simplex $\sigma < K$ is homeomorphic to closed disk, but in our case we may conclude that the open star
  \begin{equation*}
    |K| \supset \mathrm{star}(\sigma) \xrightarrow{\approx} \bigcup_{\tau: \sigma \leq \tau} j(\mathring{\Delta}^\tau) \subset S(V)
  \end{equation*}
  is \emph{diffeomorphic} to $R - \R b_\sigma$ in the smooth structure it inherits as an open subset of $S(V)$, equivariantly for $G_\sigma$.  Indeed, under the gnomonic projection the open star is identified with an open subset $\Omega \subset R - \R b_\sigma$ which is star-shaped around the origin.  The usual rescaling argument (which we learned from \cite{ball}, see there for more details) produces a diffeomorphism
  \begin{align*}
    \Omega & \to R - \R b_\sigma\\
    x & \mapsto \lambda(x) \cdot x,
  \end{align*}
  where $\lambda: \Omega \to [1,\infty)$ is a rescaling factor defined by
  \begin{equation*}
    \lambda(x) = 1 + \big(\int_0^1 \frac{dt}{\phi(tx)} \big)^2 |x|^2,
  \end{equation*}
  where $\phi: (R - \R b_\sigma) \to [0,\infty)$ is a smooth function with $\phi^{-1}(0) = (R - \R b_\sigma) \setminus \Omega$.  If we choose $\phi$ to be $G_\sigma$-equivariant (by averaging) then the diffeomorphism is also equivariant.  This proves the first regularity condition.

  The embedding $\Delta^p \hookrightarrow (R - \R b_\sigma)$ lands in a $p$-dimensional linear subspace which we will denote $T_\sigma$, and whose orthogonal complement we will denote $N_\sigma$.  Both are invariant under $G_\sigma$, and as abstract representations may be identified with the tangent and normal spaces of $j(\Delta^p) \subset S(V)$ at the barycenter.  The rescaling diffeomorphism $\mathrm{star}(\sigma) \approx R - \R b_\sigma$ then takes the open simplex $\mathring{\Delta}^\sigma \subset \mathrm{star}(\sigma)$ diffeomorphically onto $T_\sigma \subset R - \R b_\sigma$, and hence the deleted open star $\mathrm{star}(\sigma) \setminus \mathring{\Delta}^\sigma$ is taken diffeomorphically onto the open subspace
  \begin{equation}
    \label{eq:11}
    (R - \R b_\sigma) \setminus T_\sigma \approx T_\sigma \times (N_\sigma \setminus \{0\}) \approx T_\sigma \times S(N_\sigma) \times (-\infty,\infty),
  \end{equation}
  where $S(N_\sigma)$ denotes the unit sphere in $N_\sigma$ and the diffeomorphism $N_\sigma \setminus \{0\} \approx S(N_\sigma) \times (-\infty,\infty)$ is by ``polar coordinates'' $x \mapsto (\frac{x}{|x|},\log(|x|))$.  This diffeomorphism is again equivariant for $G_\sigma < G$,  where the action on $(-\infty,\infty)$ is trivial.

  A mild variation of the rescaling argument instead rescales only in the direction  $N_\sigma$ parametrized by $\mathring{\Delta}^\sigma$, and results in a diffeomorphism
  \begin{equation*}
    \mathrm{star}(\sigma) \setminus \mathring{\Delta}^\sigma \xrightarrow{\approx} \mathring{\Delta}^\sigma \times S(N_\sigma) \times (-\infty,\infty),
  \end{equation*}
  whose last coordinate $\delta: \mathrm{star}(\sigma) \setminus \mathring{\Delta}^\sigma \to (-\infty,\infty)$ extends to a continuous map
  \begin{equation*}
    \overline{\mathrm{star}}(\sigma) \setminus \partial \Delta^\sigma \to [-\infty,\infty].
  \end{equation*}
  taking the value $-\infty$ on $\mathring{\Delta}^\sigma$ and the value $+\infty$ on the remaining boundary of the closed star, which can be identified with
  \begin{equation*}
    (\mathrm{Link}(\sigma) \ast \partial \Delta^\sigma) \setminus \partial \Delta^\sigma.
  \end{equation*}

  Finally, let us review how a $G$-equivariant triangulation $j: |K| \to S(R)$ all of whose simplices are spherical may be constructed, e.g.\ using \emph{Voronoi cells}.  Start with a finite $G$-invariant subset $X \subset S(R)$ which is sufficiently fine (no ball of radius $\epsilon$ around any point in $S(R)$ is disjoint from $X$, for any given $\epsilon> 0$).  The Voronoi cell of $x \in X$ is then the subspace
  \begin{equation*}
    C_x = \{y \in S(R) \mid \mathrm{dist}(y,x) = \mathrm{dist}(y,X)\}.
  \end{equation*}
  This space is also defined by linear inequalities as follows
  \begin{equation*}
    C_x = \bigcap_{x' \in X} \{y \in S(R) \mid \langle y,x - x'\rangle \geq 0\},
  \end{equation*}
  showing in particular that it is closed and geodesically convex provided it contains no pair of antipodal points (which holds for sufficiently small $\epsilon$).  For a non-empty $A \subset X$ we write $C_A = \cap_{x \in A} C_x$ and define $J \subset 2^X \setminus \{\emptyset\}$ as
  \begin{equation*}
    J = \{A \subset X \mid C_A \neq \emptyset\}.
  \end{equation*}
  The $C_A \subset S(R)$ with $A \in J$ form the closed cells in a cell decomposition of $S(R)$.  Moreover, $A \subset A'$ if and only if $C_A \supset C_{A'}$, i.e.\ $J$ ordered by inclusion may be identified with the (opposite of) the poset of cells.  The barycentric subdivision is a simplicial complex, and we shall use that as $K$.  More precisely, we consider the map
  \begin{align*}
    N_0(J) & \to S(R)\\
    A & \mapsto \frac{b(C_A)}{|b(C_A)|},
  \end{align*}
  where $b(C_A) \in R$ is the barycenter of $C_A$, i.e.\ the average over $S(R)$ of the function $x \mapsto \chi_{C_A}(x) \cdot x$.  There is a unique extension to a map $|N_\bullet(J)| \to S(R)$ which is spherical on each non-degenerate simplex, and this map is a homeomorphism.  We then let $K$ be the ordered simplicial complex whose $p$-simplices are the non-degenerate elements of $N_p(J)$.
\end{rmrk}

\subsection{Finishing the proof}\label{finishing_section}

We finish the proof of \cref{eqdeloop}, asserting that 
$$\alpha_R\colon \psi_d(V,W)\to \Omega^R \psi_d (V, W+R)$$
is an equivariant equivalence when $\dim(W^G) > d$.

Using \cref{bottommap} we know that the horizontal maps in \eqref{segalannulus} are equivariant equivalences, so the diagram looks like
\begin{equation}\label{segal_annulus_mapping_models}
 \begin{tikzcd}
        \PsiRd (\overline{D_{1}(R)}_\epsilon) \arrow[d] \arrow[r, "\simeq"] & \Map \big (\overline{D_1(R)}, \PsiRd(D_\epsilon(R)) \big) \arrow[d] \\
        \PsiRd (S(R)_{\epsilon}) \arrow[r, "\simeq"]  & \Map \big (S(R), \PsiRd(D_\epsilon(R)) \big).
 \end{tikzcd}
 \end{equation}
The right vertical map is an equivariant fibration with fiber $\Omega^R \psi_d (V, W+R)$.  The left vertical map is not a fibration, but the homotopy fibers may be analyzed as in the proof of Proposition~\ref{homotopy_sheaf}.  That proof achieves a bit more than proving the square is $G$-homotopy cartesian, it also determines the fibers in both vertical and horizontal directions up to equivariant weak equivalence: in the notation of~(\ref{zigzagsquare}) they are $X$ and $Y$, respectively.  In particular, when $O_1 = \overline{D_1(R)}_\epsilon$ and $O_2 = O_1 \cap O_2 = S_1(R)_\epsilon$, the front faces of the cubes in \cref{zigzagsquare} become
\begin{equation}\label{segal_annulus_bar_models}
\begin{tikzcd}
        B(X, \mathcal{M}, \ast) \arrow{d}{p_\mathcal{M}} & 
        B\mathcal{P}_X \arrow{l}[swap]{\simeq} \arrow{r}{\simeq} \arrow{d}{p_\mathcal{P}} & \PsiRd(\overline{D_1(R)}_\epsilon) \arrow{d}
    \\
        B(\ast, \mathcal{M}, \ast) & 
        B\mathcal{P} \arrow{l}[swap]{\simeq} \arrow{r}{\simeq} &  
        \PsiRd(S_1(R)_\epsilon).
\end{tikzcd}
\end{equation}
For the function $\delta \colon O_1 \cap O_2 \to (-\infty, +\infty)$ in \cref{homotopy_sheaf} we can take distance from the origin $O_1 \cap O_2 \to (1-\epsilon, 1+\epsilon)$ and rescale to $(-\infty, +\infty)$, then $X$ is identified with the space of pairs $(M,a)$ with $a \in (1-\epsilon,1+\epsilon)$ and $M \in \PsiRd(O_1)$ has $M \subset \pi_R^{-1}(D_a(R))$.  There is an embedding $\psi_d(V,W) \hookrightarrow X$, identifying $\psi_d(V,W)$ with the subspace where $a=1$, and this embedding is obviously an equivariant weak equivalence.  Therefore~(\ref{segal_annulus_bar_models}) implies an equivariant weak equivalence $\psi_d(V,W) \simeq \Omega^R \psi_d (V, W+R)$.

We are now almost done: we wanted to show that the particular map $\alpha_R: \psi_d(V,W) \to \Omega^R \psi_d(V,W+R)$ is an equivalence, what remains is to identify $\alpha_R$ with the weak equivalence just established, as morphisms in the $G$-equivariant homotopy category.
Let
$\psiRd(D_a(R))$ denote the subspace of manifolds $M \in \psi_d(V, W + R)$ such that $M \subset \pi_R^{-1}(D_a(R))$. Then in the case above, $X$ is the space of pairs $(M, a)$ such that $a \in (1-\epsilon, 1+\epsilon)$ and $M \in \psiRd(D_a(R))$. The fibers of the middle and right hand vertical maps in \eqref{segal_annulus_bar_models} are both $\psiRd(D_{1-\epsilon}(R))$.

Composing the two diagrams above and investigating the fibers and homotopy fibers, we end up with the following diagram.
\[
\begin{tikzcd}
        X \arrow[d, "\simeq"] \arrow[r, leftrightarrow, "\simeq"]&
        \psiRd (D_{1-\epsilon}(R)) \arrow[r, "\mathrm{scan}"] \arrow[d] &
        \Omega^R \PsiRd(D_\epsilon(R)) \arrow[d, "\simeq"]
    \\
        \hofib(p_\mathcal{M}) \arrow[d] &
        \hofib(p_\mathcal{P}) \arrow[d] \arrow[r] \arrow{l}[swap]{\simeq} &
        \hofib(\text{res}) \arrow[d]
    \\
        B(X, \mathcal{M}, \ast) \arrow{d}{p_\mathcal{M}} & 
        B\mathcal{P}_X \arrow{l}[swap]{\simeq} \arrow[r, "\simeq"] \arrow{d}{p_\mathcal{P}} & 
        \Map \big (\overline{D_1(R)}, \PsiRd(D_\epsilon(R)) \big) \arrow{d}{\text{res}} 
    \\
        B(\ast, \mathcal{M}, \ast) & 
        B\mathcal{P} \arrow{l}[swap]{\simeq} \arrow[r, "\simeq"] &  
        \Map \big (S(R), \PsiRd(D_\epsilon(R)) \big)
\end{tikzcd}
\]
Here each column has the form $\fib(p) \to \hofib(p) \to A \xrightarrow{p} B$, and the horizontal arrows are (composites of) maps defined above. This shows that $\mathrm{scan}$ is an equivariant equivalence.

The following lemma allows us to relate our results to the map $\alpha_R$, and thus finishing the proof of \cref{eqdeloop}.

\begin{lemma}\label{homotopyfibers}
    There is a diagram of the form
    \[
    \begin{tikzcd}
        \psi_d(V,W) \ar[d, "\simeq"] \ar[r, "\alpha_R"] & \Omega^R \psi_d (V, W+R) \ar[d, "\simeq"] \\
        \psiRd (D_{1-\epsilon}(R)) \ar[r, "\mathrm{scan}"] & \Omega^R \PsiRd(D_\epsilon(R)),
    \end{tikzcd}
    \]
    commuting up to equivariant homotopy.
\end{lemma}
\begin{proof}
  Let us first point out that $\Omega^R$ in the upper right corner means pointed maps from $S^R$, the one-point compactification of $R$, while in the lower right corner it means pointed maps from $\overline{D_1(R)}/S(R) = D_1(R) \cup\{\infty\}$.  We compare these using the equivariant homeomorphism $D_1(R) \cup\{\infty\} \to S^V$ extending $r \mapsto -(1 - \|r\|)^{-1} r$,
  and let $\psi_d(V,W + R) \to \mathcal{F}(D_\epsilon(R))$ be $M \mapsto M \cap \pi_R^{-1}(D_\epsilon(R))$.  
  We take the vertical map on the right to be given by pre- and post-composing with these.  For $M \in \psi_d(V,W)$, the right vertical map then takes $\alpha_R(M)$ to the extension of
  \begin{align*}
    D_1(R) & \to \mathcal{F}(D_\epsilon(R))\\
    r & \mapsto \varphi_r^{-1}(M) \cap \pi_R^{-1}(D_\epsilon(R)),
  \end{align*}
  where $\varphi_r(x) = x + (1 - \|r\|)^{-1} r$ for $x \in \pi_R^{-1}(D_\epsilon(R))$.

  As the left vertical map we take the scaling $M \mapsto (1-\epsilon) M$, which is taken to the extension of
  \begin{align*}
    D_1(R) & \to \mathcal{F}(D_\epsilon(R))\\
    r & \mapsto \big(t_r^{-1}((1 - \epsilon) M)\big) \cap \pi_R^{-1}(D_\epsilon(R)),
  \end{align*}
  by the scanning map; recall that $t_r(x) = x + r$.  We may now define an equivariant path between these two which at time $t \in [0,1]$ is the extension of $r \mapsto T_{r,t}^{-1}(M)$, where
  \begin{equation*}
    T_{r,t}(x) = \frac{1 - t \epsilon}{1 - \epsilon} \bigg( x + \frac{1 + (1-t)\|r\|}{1 + \|r\|} r\bigg) \qedhere
  \end{equation*}
\end{proof}

\section{Spaces of manifolds and affine Grassmanians}\label{sec:affine_grassmanian}

As the final ingredient to \cref{centralthm}, we prove \cref{affine_grassmanian}. That is, we construct an equivariant equivalence $\psi_d(V, V) \simeq \MTO_d(V)$ between the space of unbounded manifolds in $V$ and the affine Grassmanian of $d$-planes in $V$.

Recall that
$$ \MTO_d(V)=\mathrm{Th} \left (
  \begin{tikzcd} 
 \xi_V^\bot \arrow{d}{} \\
 \mathrm{Gr}_d(V)
 \end{tikzcd}
 \right )
 $$
\begin{defn}
    Let $q\colon \MTO_d(V) \to \psi_d(V, V)$ be the equivariant map defined as follows. A point $L \in \xi_V^\bot = \MTO_d(V)\setminus \{ \infty \}$ can be identified with an affine $d$-plane in $V$, which is a $d$-dimensional submanifold, hence a point in $\psi_d(V, V)$. We map $\infty \in \MTO_d(V)$ to $\varnothing \in \psi_d(V, V)$, which defines a continuous map because of the way we defined the topology on $\psi_d(V, V)$.
\end{defn}

\begin{lemma}
    The map $q\colon \MTO_d(V) \to \psi_d(V, V)$ is an equivariant equivalence.
\end{lemma}
\begin{proof}
    The proof is identical to that of \cite[Lemma 6.1]{gal11}. We split $\psi_d(V, V)$ as the pushout of open sets $U_0 \xleftarrow{} U_{01} \xrightarrow{} U_1$. Here $U_0 \subset \psi_d(V, V)$ is the open subset of manifolds $M\subset V$ such that $0\not \in M$. The subset $U_1 \subset \psi_d(V, V)$ consists of manifolds $M \subset V$ with a unique non-degenerate closest point to the origin, and $U_{01} = U_0 \cap U_1$. 
    The subsets $U_0, U_1, U_{01}$ are $G$-invariant open subsets, and the restrictions of $q$
    \begin{align*}
        q^{-1}(U_0) \to U_0 \\
        q^{-1}(U_1) \to U_1 \\
        q^{-1}(U_{01}) \to U_{01}
    \end{align*}
    are all equivariant equivalences. The spaces $q^{-1}(U_0)$ and $U_0$ are both equivariantly contractible. For $M\in U_1^H$, notice that the unique closest point $p\in M$ must be contained in $V^H$ (since any point in the orbit of $p$ is closest to the origin), hence the deformation retraction described in \cite[Lemma 6.1]{gal11} remains inside the fixed point space. The same applies to the restriction to $U_{01}$.
\end{proof}

\section{Tangential structures}\label{sec:tangential_structures}

In this section we briefly discuss a variant of \cref{centralthm} involving tangential structures. For trivial $G$ this is discussed in \cite[Section 5]{gmtw}.

Let $M$ be a $d$-dimensional smooth manifold and $G$ a finite group. Write $\GL_d = \GL_d(\mathbf{R})$, and let $\Fr(M)$ denote the frame bundle of $M$. That is, $\Fr(M) \to M$ is the principal $\GL_d$-bundle of bases of the tangent bundle $TM\to M$. Let $\GL_d$ act on $\Fr(M)$ on the left by change of basis.

\begin{defn}
  Let $\Theta$ be a space with a left action of $\GL_d \times G$. A $\Theta$-structure $\ell$ on $M$ is a $\GL_d$-equivariant map $\ell\colon \Fr(TM) \to \Theta$. If $G$ acts smoothly from the left on $M$, there is an induced left action of $G$ on $\Fr(TM)$ by differentiating the action on $M$.  This action commutes with the $\GL_d$ action so we view $\Fr(TM)$ as a left $(\GL_d \times G)$-space, and say that the $\Theta$-structure on $M$ is \emph{equivariant} if $\ell$ is $\GL_d \times G$-equivariant.
\end{defn}

Fix the dimension $d\geq 0$, a finite group $G$ and a $(\GL_d \times G)$-space $\Theta$. Let $\Gamma = \GL_d \times G$. The following is the definition of the equivariant cobordism category of manifolds with tangential structures.

\begin{defn}
    Let $\mathcal{C}_\Theta^{G}$ be the following topologically enriched category. Objects are pairs $(M, \ell)$, where $M\in \mathrm{Ob}(\mathcal{C}_d^G)$ as in \cref{equivariant_cobordism_category} and $\ell\colon \Fr(TM\oplus 1)\to \Theta$ is an equivariant $\Theta$-structure. 
    
    The morphism space between $(M_0, \ell_0)$ and $(M_1, \ell_1)$ is formed by triples $(N, r, \ell)$ where $(N, r) \in \mathcal{C}_d^G(M_0, M_1)$ and $\ell\colon \Fr(TN) \to \Theta$ is an equivariant $\Theta$-structure that restricts to $\ell_0$ and $\ell_1$ on $\partial N$.
    
    The homotopy type of morphism spaces may be described as
    $$
    \mathcal{C}_\Theta^G(M_0, M_1) \simeq \coprod\limits_L B\Diff_\Theta^G (L, \partial L), 
    $$
    where the disjoint union is over smooth equivariant cobordisms
    $L$ 
    between $M_0$ and $M_1$, one in each diffeomorphism class relative to $\partial L = M_0 \amalg M_1$. For such $L$, we write $\Diff^G(L, \partial L)$ for the topological group of equivariant diffeomorphisms that restrict to the identity in a neighborhood of $\partial L$, and $B\Diff_\Theta^G(L, \partial L)$ denotes the homotopy quotient
    $$
        B\Diff_\Theta^G(L, \partial L) = \Map^\partial_{G} \left(\Fr(TL), \Theta \right ) \hquotient \Diff^G(L, \partial L), 
    $$
    where $\Map^\partial_{G} \left(\Fr(TL), \Theta \right )$ is the space of equivariant $\Theta$ structures on $L$, fixed near the boundary (omitting the boundary conditions from the notation).
\end{defn}

\newcommand{\basespace}{B}

Now we introduce the equivariant orthogonal spectrum $MT\Theta$. For an orthogonal $G$-representation $V$ let $\basespace(V) = (\mathrm{St}_d(V) \times \Theta)/\GL_d$, where $\mathrm{St}_d(V)$ is the Stiefel manifold of $d$-frames in $V$.  The unique map $\Theta \to \{\ast\}$ induces a $G$-equivariant map $\vartheta_V \colon \basespace(V) \to \mathrm{Gr}_d(V)$, and we let $\vartheta_V^{\ast} \xi^\bot$ be the pullback of the complement of the tautological bundle over $\mathrm{Gr}_d(V)$.

\begin{defn}
    Define the orthogonal $G$-spectrum $MT\Theta$ by
    \[
    MT\Theta (V) = \mathrm{Th} \left (
    \begin{tikzcd}
        \vartheta_V^{\ast} \xi^\bot \arrow{d}{} \\
        \basespace(V)        
    \end{tikzcd}
    \right ).
    \]
    The structure maps are defined analogously to \cref{MTOspectrum}.
\end{defn}

We have a version of our main Theorem~\eqref{thm:mainthm}, generalized to include tangential structures. The special case $\Theta = \{ \ast \}$ recovers the original (unoriented) statement.

\begin{thm}\label{mainthmtangential} The classifying space of the equivariant cobordism category with $\Theta$-structures is weakly equivalent to the fixed point space of the shifted infinite loop space of $MT\Theta$:
  \begin{equation}
    \label{eq:7}
    B\mathcal{C}_\Theta^G \simeq \left ( \Omega^{\infty-1} MT\Theta \right )^G
  \end{equation}
\end{thm}

Similarly to the unoriented case, this is the fixed point level statement of the following.

\begin{thm}
  There is an equivariant equivalence
  \begin{equation}
    \label{eq:4}
    B\mathcal{C}_\Theta(\mathcal{U}_G) \simeq \Omega^{\mathcal{U}_G-1} MT\Theta.
  \end{equation}
\end{thm}

The embedded cobordism categories $\mathcal{C}_\Theta(V)$ are defined similarly to the unoriented case. We also have analogues of spaces of manifolds.

\begin{defn}
    For $V$ a finite dimensional $G$-representation, let $\Psi_\Theta (V)$ be the set of pairs $(M, \ell)$ where $M\in \Psi_d (V)$ and $\ell$ is a $\Theta$-structure on $M$.
\end{defn}

The topology on $\Psi_\Theta (V)$ is defined as for $\Psi_d(V)$, now modeled on the spaces 
$$\Gamma_c (NM) \times \Map_{\GL_d}(\Fr(TM), \Theta). $$ 
Define the left conjugation action of $G$ on $\Psi_\Theta (V)$ by $g(M, \ell) = (gM, g\ell)$, where $gM\subset V$ is the image of $M\subset V$ under $G$, and $g\ell \colon \Fr(T(gM)) \to \Theta$ is defined by $g\ell (b) = g \left (\ell(g^{-1} b) \right )$, i.e.\ so that
\[
\begin{tikzcd}
    \Fr(TM) \arrow{r}{\ell}\arrow{d}{g} & \Theta \arrow{d}{g} \\
    \Fr(T(gM)) \arrow{r}{g\ell} & \Theta
\end{tikzcd}
\]
commutes.  With this definition the fixed point space $\Psi_\Theta (V)^G$ consists of manifolds $M$ equivariantly embedded in $V$, with equivariant $\Theta$-structures.

Our proof of the main theorem generalizes to the case of tangential structures without any substantial changes, since we can canonically carry the $\Theta$-structures along diffeomorphisms and restrict them to submanifolds when necessary.

\newcommand{\Or}{\mathrm{Or}}  

\begin{rmrk}\label{rem:htpy-type-of-Theta}
  Both
  the cobordism category $\mathcal{C}_\Theta$ and the spectrum $MT\Theta$ depend functorially on the $(\GL_d \times G)$-space $\Theta$, and the equivalences~\eqref{eq:7} and~\eqref{eq:4} are through zig-zags of natural transformations.  A $(\GL_d \times G)$-equivariant map $\Theta \to \Theta'$ induces a weak equivalence of domains and codomains of~\eqref{eq:7} and~\eqref{eq:4} when the induced maps $\Theta^H \to (\Theta')^H$ are non-equivariant weak equivalences for all $H \leq G$.  Thus, in the terminology of equivariant homotopy theory, the functoriality in $\Theta$ is ``Borel equivariant'' for the $\GL_d$-action and ``genuine equivariant'' for the $G$-action.

  The equivariant homotopy type of $\Theta$ is perhaps easier to analyze via the associated object~\eqref{eq:12} below, as follows.  First, pass to the $G$-equivariant space $B = B(\mathcal{U}) = (\mathrm{St}_d(\mathcal{U}) \times \Theta)/\GL_d$, which comes with a $G$-equivariant map $B \to \mathrm{St}_d(\mathcal{U})/\GL_d = \mathrm{Gr}_d(\mathcal{U})$.  No essential information is lost by this process, since the canonical map
  \begin{equation*}
    \Theta \to \mathrm{St}_d(\mathcal{U}) \times_{\mathrm{Gr}_d(\mathcal{U})} B
  \end{equation*}
  is a $(\GL_d \times G)$-equivariant map inducing a weak equivalence of fixed points for any $H < G$.  Therefore the input to Theorem~\ref{mainthmtangential} could equivalently be specified as a $G$-equivariant space over $\mathrm{Gr}_d(\mathcal{U})$.  The fixed point space $B^H = \Map_G(G/H,B)$ comes with a map
  \begin{equation*}
    B^H \to (\mathrm{Gr}_d(\mathcal{U}))^H,
  \end{equation*}
  and as usual both may be regarded as functors from the orbit category $\Or(G)$, whose morphisms are $G$-equivariant maps $G/H \to G/H'$.  By Elmendorff's theorem, no essential loss of information is lost in replacing the equivariant space $B$ by the functor $G/H \mapsto B^H$.    We then arrive at an object
  \begin{equation}\label{eq:12}
    (G/H \mapsto B^H) \in \mathrm{Fun}(\Or(G),\mathrm{Top})_{/((G/H) \mapsto \mathrm{Gr}_d(\mathcal{U})^H)},
  \end{equation}
  encoding the equivariant homotopy type of $\Theta$ in the sense explained.

  The codomain of~\eqref{eq:12} denotes the over category in the category of functors.
  For later use let us point out that the map  
  \begin{align*}
    \mathrm{Gr}_d(\mathcal{U})^H & \to \coprod_{n=0}^d \mathrm{Gr}_n(\mathcal{U})\\
    V & \mapsto V^H
  \end{align*}
  defines a natural transformation of functors of $(G/H) \in \Or(G)$, whose codomain is the constant functor.  In particular it is invariant under the group of automorphisms of $G/H \in \Or(G)$, which is the Weyl group $W_G(H) = N_G(H)/H$.  Hence we get an induced map of spaces
  \begin{equation*}
    \coprod_{(H)} (B^H)_{hW_G(H)} \to
    \coprod_{(H)} \big(\mathrm{Gr}_d(\mathcal{U})^H\big)_{hW_G(H)} \to
    \coprod_{n=0}^d \mathrm{Gr}_n(\mathcal{U})
    \simeq \coprod_{n=0}^d BO_n,
  \end{equation*}
  where the disjoint unions are over subgroups $H < G$, one in each conjugacy class.  (More canonically, the domain may be written as $\hocolim_{G/H \in \Or(G)} B^H$.)
\end{rmrk}

\section{Examples and applications}

We discuss some special cases of \cref{mainthmtangential} for various tangential structures, as well as some consequences and potential applications.

\subsection{Some special cases}
\label{sec:examples}

\begin{example}[Orientation reversing action] Let $G = \mathbf{Z}/2$ and $\Theta = \{-1, +1\}$, where $\GL_d$ acts on $\Theta$ via the sign of the determinant, and $G$ acts on $\Theta$ by transposition. In this case $\Theta$-manifolds are manifolds with an orientation reversing involution.
\end{example}

In the special case $d=2$ of the above example, our \cref{mainthmtangential} recovers the results of Nisan Stiennon on characteristic classes of real curves (\cite{nisan}).

\begin{example}[Unoriented manifolds] If $\Theta = \{ \ast \}$ we recover our original statement in Theorem~\eqref{thm:mainthm} about unoriented manifolds.
\end{example}

It is interesting to consider the content of \cref{mainthmtangential} in the case of $0$-manifolds. Let $\mathcal{F}_G$ denote a skeleton of the category of finite $G$-sets,  and for an object $A\in \mathcal{F}_G$ let $\Sigma_A = Aut(A)$ denote the group of equivariant bijections of $A$.  Disjoint union of $G$-sets gives rise to a symmetric monoidal structure on $\mathcal{F}_G$, and we may arrange that the underlying monoidal structure is strict (associators and units are identities).
Then the space
$$\mathcal{M} = \coprod_{A \in \mathcal{F}_G} B\Sigma_A,$$
inherits the structure of a topological monoid, which makes it a model for the cobordism category $\mathcal{C}_0^G$.

The spectrum $\MTO_0 = \mathbf{S}_G$ is the $G$-equivariant sphere spectrum.  As a special case of \cref{mainthmtangential} we obtain the following theorem, due to Guillou and May.

\begin{thm}[Equivariant Barratt-Priddy-Quillen Theorem, {\cite[Section 6.1]{equi_bpq}}]
    There is a weak equivalence of spaces
    $$
    \Omega B \left ( \coprod_{A \in \mathcal{F}_G} B\Sigma_A \right) \simeq \left ( \Omega^\infty \mathbf{S}_G \right )^G\!.
    $$
\end{thm}

More generally we can consider the case of $0$-manifolds for any $G$-space $\Theta$, and get a weak equivalence
\begin{equation}\label{eq:15}
\Omega B \left ( \coprod_{[A] \in \mathcal{F}_G} \Map_G(A, \Theta) \hquotient \Sigma_A \right) \simeq \left( \Omega^\infty\Sigma^\infty \Theta_{+}\right)^G\!,
\end{equation}
where $\Map_G(A, \Theta)$ is the space of $G$-equivariant maps and $\Map_G(A, \Theta) \hquotient \Sigma_A$ is the homotopy quotient 
$$
\Map_G(A, \Theta) \times_{\Sigma_A} E\Sigma_A.
$$
The domain of~\eqref{eq:15} may up to homotopy be identified as the group completion of a \emph{free} $E_\infty$ space, using that a finite $G$-set splits canonically as a disjoint union of orbits.  We leave it as an exercise
to relate this observation to the tom Dieck splitting, cf.\ \cite[Section 6.2]{equi_bpq}.

Finally, for any orthogonal representation $\rho$ of $G$, our main theorem gives an interpretation of the infinite loop space associated to an inverse representation sphere $\mathbf{S}_G^{-\rho}$, i.e.\ the orthogonal spectrum
$\mathbf{S}_G^{-\rho} = \mathcal{J}_G(\rho, -)$ (see \cite[Definition 4.6]{MandellMay}).
\begin{example}[Action on frames] Let $\rho\colon G \to \GL_d$ be a representation of $G$, and let $\Theta = \GL_d$ with left $(\GL_d \times G)$-action given
  by $(A,g)B = AB\rho(g^{-1})$.  In this case an equivariant $\Theta$-structure on a $G$-manifold $M$ is an equivariant framing: an equivariant isomorphism $M \times \R^d \to TM$ of vector bundles over $M$, where $G$ acts on $\R^d$ via $\rho$.

  Note that with this action,
    $$\Theta \times_{\GL_d} \mathrm{St}_d(V) \cong \mathcal{L}(\rho, V),$$
    so $MT\Theta$ is isomorphic to the spectrum $\mathbf{S_G}^{-\rho} = \mathcal{J}_G(\rho, -)$.  Hence we get an equivariant weak equivalence
    \begin{equation*}
      \Omega^\infty \mathbf{S}_G^{-\rho} \simeq \Omega B \mathcal{C}_\Theta.
    \end{equation*}
\end{example}

\subsection{Rational cohomology and characteristic classes}
\label{sec:rati-cohom-char}

A
main application of cobordism categories in the non-equivariant case is to moduli spaces of manifolds, see the recent survey \cite{haynesbook} and the references therein.  Let us end this section with some preliminary remarks about a possible application of Theorems~\ref{thm:mainthm} and \ref{mainthmtangential} to moduli spaces of \emph{equivariant} manifolds, focusing on closed manifolds for simplicity.  We thank Markus Hausmann for helpful discussions regarding stable equivariant rational homotopy theory.

The starting point is the observation that a closed $d$-manifold may be regarded as an endomorphism of the $(d-1)$-manifold $\emptyset$, leading by Theorem~\ref{mainthmtangential} to a map
\begin{equation}\label{eq:1x}
  \coprod_L B\Diff^G_\Theta(L) \simeq \mathcal{C}^G_\Theta(\emptyset,\emptyset) \to \Omega B\mathcal{C}^G_\Theta \simeq (\Omega^\infty MT\Theta)^G.
\end{equation}
As explained in Section~\ref{sec:equivariant-bundles} the space $B\Diff^G(L)$ is a classifying space for smooth equivariant bundles whose fibers are diffeomorphic to $L$, and similarly $B\Diff^G_\Theta(L)$ is a classifying space for such bundles equipped with a fiberwise equivariant $\Theta$-structure.  Any cohomology class $c \in H^*((\Omega^\infty MT\Theta)^G)$ may therefore be pulled back along~(\ref{eq:1x}) to give a characteristic class for such bundles.

The ring of characteristic classes arising by this construction is most conveniently expressed in terms of the objects associated to $\Theta$ in Remark~\ref{rem:htpy-type-of-Theta}.  As explained there, the $H$-fixed points of the $G$-space $B = (\mathrm{St}_d(\mathrm{U}) \times \Theta)/\GL_d$ come with a map
\begin{equation*}
  \theta_H: (B^H)_{hW_G(H)} \to \coprod_{n = 0}^d BO_n.
\end{equation*}

\begin{lemma}
  Let $\theta_H^* \gamma$ denote the vector bundle on $(B^H)_{hW_G(H)}$ classified by the map $\theta_H$ above, and let $(B^H_{hW_G(H)})^{-\theta_H}$ denote the Thom spectrum of the inverse virtual bundle.  Then there is a canonical rational equivalence of spectra
  \begin{equation}\label{eq:3x}
    (MT\Theta)^G \to \prod_H (B^H_{hW_G(H)})^{-\theta_H},
  \end{equation}
  where the product is indexed by subgroups $H < G$, one in each conjugacy class.
\end{lemma}

A special case of the following corollary was originally part of joint work between the second author and Bena Tshishiku.  We thank him for several helpful discussions.
\begin{coro}
  For each $n \in \{0, \dots, d\}$ write $X_{H,n} = \theta_H^{-1}(BO_n)$, so that $\theta_H$ restricts to a map $\theta_{H,n}: X_{H,n} \to BO_n$.  Let us write $\Q^{w_1(\theta_{H,n})}$ for the local system corresponding to the orientation character of $\theta_{H,n}$.  For each $c \in H^{n+i}(X_{H,n};\Q^{w_1(\theta_{H,n})})$ there is an associated class
  \begin{equation*}
    \kappa_{H,c} \in H^i((\Omega^\infty MT\Theta)^G;\Q).
  \end{equation*}
  If $B^H$ has finite type for all $H$, then these classes induce an isomorphism
  \begin{equation*}
    \Q[\kappa_{H,c}] \to H^*((\Omega^\infty MT\Theta)^G_0;\Q)
  \end{equation*}
  whose domain is the free graded-commutative algebra of the classes $\kappa_{H,c}$, where $H \leq G$ range through one subgroup in each conjugacy class, $n$ ranges through integers between $0$ and $d$, and $c$ ranges through a homogeneous basis for $H^{>n}(X_{H,n};\Q^{w_1(\theta_{H,n})})$.
\end{coro}
\begin{proof}
  The map on infinite loop spaces induced by~(\ref{eq:3x}) induces an isomorphism on rationalized homotopy groups, hence it is a rational equivalence of spaces when restricted to a map between path connected spaces.  The corollary then follows from Thom isomorphism and the usual relationship between rational cohomology of a of finite type rational spectrum and its associated infinite loop space.

    For any $G$-equivariant spectrum $X$ and subgroup $H \leq G$, recall that the \emph{geometric fixed points} are defined as follows.  If $E\mathcal{P}_H$ is a $G$-CW complex with
  $(E\mathcal{P}_H)^K \simeq \{\ast\}$ when $K$ is conjugate to a proper subgroup of $H$ and $(E\mathcal{P}_H)^K = \emptyset$ otherwise, we define $\widetilde{E\mathcal{P}}_H$ as the mapping cone of $(E\mathcal{P}_H)_+ \to S^0$ and the geometric fixed points as
  \begin{equation*}
    \Phi^H(X) = (X \wedge \widetilde{E\mathcal{P}_H})^H = F_G((G/H)_+,X \wedge \widetilde{E\mathcal{P}_H}),
  \end{equation*}
  where $F_G$ denotes the spectrum of $G$-equivariant maps.  There is an evident action of $W = W(H) = \Map_G(G/H,G/H) = N_G(H)/H$, and a $W$-equivariant map $X^H \to \Phi^H(X)$.  We now consider the composition
  \begin{equation*}
    X^G \to X^H \to \Phi^HX \to (\Phi^H X)_{hW},
  \end{equation*}
  where the last map is the projection to the homotopy orbits.  We combine these maps to a natural transformation
  \begin{equation}\label{eq:2x}
    X^G \to \prod_H (\Phi^H X)_{hW(H)},
  \end{equation}
  which we claim is a rational equivalence for any $G$-equivariant spectrum $X$.  To see this it suffices to consider the cases $X = \Sigma^\infty_+ G/K$ for subgroups $K \leq G$.  Then $\Phi^H(\Sigma^\infty_+ G/H) = \Sigma^\infty_+(G/H)^K$ and we must take homotopy colimit as $G/K$ runs through the category of $G$-orbits.  Now homotopy colimit commutes with suspension spectrum and
  \begin{equation*}
    \hocolim_{G/K} (G/H)^K \simeq B(\text{$G$-orbits over $G/H$}) \simeq B(\text{$H$-orbits}) = \coprod_{H/K} \Sigma^\infty_+ BW_H(K),
  \end{equation*}
  where the disjoint union is over one $K \leq H$ in each conjugacy class.  By the tom Dieck splitting, this agrees with the homotopy type of $(\Sigma^\infty_+ (G/H))^H$, but the map is not an equivalence integrally (indeed, the natural transformation~(\ref{eq:2x}) is defined for any $X$, but the tom Dieck splitting only holds for suspension spectra).  On the level of $\pi_0$ it may be identified with the ring homomorphism
  \begin{align*}
    A(H) &\to \prod_{H/K} \Z\\
    X &\mapsto (|X^K|)_{H/K},
  \end{align*}
  which is well known to be a rational isomorphism.  For $X = \Sigma^\infty_+ (G/H)$ both domain and codomain of~(\ref{eq:2x}) have vanishing rational homotopy in all non-zero degrees, so the map is a rational equivalence in this case.  Since the class of spectra for which it is a rational equivalence is closed  under suspension and desuspension, mapping cones, and filtered homotopy colimits, it must be a rational equivalence for all $X$.  See also \cite[Appendix A]{GreenleesMay} and \cite[Section V]{LewisMaySteinberger}. 

  It is well known how geometric fixed points $\Phi^H$ behave on suspension spectra, namely $\Phi^H(\Sigma^\infty_+ A) = \Sigma^\infty _+ A^H$.  This fact generalizes to Thom spectra: the geometric fixed points of the Thom spectrum of an equivariant bundle over $A$ is the Thom spectrum of the non-equivariant bundle over $A^H$ formed by the fixed points, and a similar fact holds for Thom spectra of inverse vector bundles. 
  This proves the claim.  
\end{proof}

For example, we can take $d=2$, $\Theta = \{\pm 1\}$ on which $\GL_2$ acts by the determinant and $G$ acts trivially.  The corresponding equivariant spectrum $MT\Theta$ is an equivariant version of the spectrum $MTSO_2$ from e.g.\ \cite{gmtw}.
The corresponding cobordism category has morphisms oriented surfaces equipped with orientation preserving $G$-action.  Then $\mathrm{St}_2(\mathcal{U}) \times_{\GL_2} \Theta$ is a model for $B_GSO_2$ and the $H$-fixed points split as a disjoint union according to isomorphism type of representation $\rho: G \to SO_2$.  For a representation $\rho: H \to SO_2$ the invariants $\rho^H$ are zero-dimensional, unless $\rho$ is trivial, in which case $\rho^H = \rho$.  Since $SO_2$ is abelian the automorphism group of any representation is $SO_2$ and all Weyl groups are trivial.  It follows that in the notation above, for each subgroup $H \leq G$
\begin{align*}
  X_{H,2} &= BSO_2 \\
  X_{H,1} & = \emptyset\\
  X_{H,0} &= \mathrm{Hom}^{\mathit{non-triv}}(H,SO_2) \times BSO_2.
\end{align*}
In this case we get a rational equivalence
\begin{equation*}
  MT\Theta^G \xrightarrow{\simeq_\Q} \bigvee_{H < G} \bigg(MTSO_2 \vee \bigvee_{\substack{H \to SO_2\\\text{non-zero}}} \Sigma^\infty_+ BSO_2 \bigg).
\end{equation*}
The classes $e^{i+1} \in H^{2i+2}(X_{H,2})$ give rise to classes $\kappa_{H,e^{i+1}} \in H^{2i}(\Omega^\infty MT\Theta^G;\Q)$ for each $H < G$, while for each non-trivial homomorphism $\rho: H \to SO_2$ the class $\rho \otimes e^i \in H^{2i}(X_{H,0};\Q)$ gives rise to a class $\kappa_{H,\rho \otimes e^i} \in H^{2i}(\Omega^\infty MT\Theta^G;\Q)$.  As explained, each of these give rise to characteristic classes of bundles of $G$-equivariant surfaces, which seem to be the $G$-equivariant analogues of the Miller--Morita--Mumford classes.  These characteristic classes seem worth a further study.

More
generally, it seems interesting to investigate the properties of the maps $B\Diff_\Theta^G(L) \to \Omega^\infty MT\Theta^G$.  For example, whether there are good cases in which these maps induce homology isomorphisms in a range, or, less ambitiously, whether all the characteristic classes above are detected on some $L$.  Based on experience from the non-equivariant case (such as \cite{EbertCrelle}) this may be more reasonable in cases where $X_{H,n} = \emptyset$ when $n$ is odd (for instance when $\Theta = \GL_{2d}(\R)/\GL_d(\C)$ with trivial $G$ action).

\section{Relationship to equivariant bordism} \label{sec:equi_bordism}

\subsection{Unoriented equivariant bordism groups}
Let $\mathcal{N}_d^G$ denote the geometric cobordism group of unoriented $d$-dimensional $G$-manifolds. As shown in \cite[Satz 5]{tom_dieck}, there is an orthogonal $G$-spectrum $mO$ and an isomorphism
\begin{align}\label{bordism}
    \mathcal{N}_d^G\cong \pi_d^G (mO)
\end{align}
for finite $G$.  After recalling the definition of $mO$ following \cite[Section 5]{tom_dieck} (see also \cite[Sections 6.1 and 6.2]{global}), we explain how this isomorphism relates to our main result.  We thank Gunnar Carlsson for helpful correspondence about equivariant bordism. 

\begin{defn}\label{mO}
  Let $mO$ be the orthogonal $G$-spectrum defined by
  \begin{align*}
    mO(V)=
    \mathrm{Th} \left (
    \begin{tikzcd} 
      \xi \arrow{d}{} \\
      \mathrm{Gr}_{|V|}(V\oplus \mathbf{R}^{\infty})
    \end{tikzcd}
    \right ),
  \end{align*} 
  where $\xi$ is the tautological $|V|$-dimensional bundle over the Grassmanian.
\end{defn}

\begin{rmrk}
  We emphasize that
  this is different from the spectrum $MO_G$ of \cite[XV.2]{alaska} defined in terms of Thom spaces over $\mathrm{Gr}_{|V|}(V\oplus V)$.  Both are genuine $G$-equivariant spectra, and $MO_G$ is a localization of $mO$, as explained in \cite[Corollary 6.1.35]{global} (in the closely related setting of ``global spectra'').  The colimit~(\ref{filtration}) below expresses the spectrum $mO$ as a kind of localization of the equivariant spectra $\MTO_d$.
\end{rmrk}

Recall
that for an orthogonal $G$-spectrum $E$ and a $G$-representation $W$ we write $\sh_W E$ for the shifted spectrum given by $\sh_W E (V)= E(V\oplus W)$ (in fact $\sh_W E \simeq \Sigma^W E$).  We will abbreviate $\sh_{\mathbf{R}^d}$ as $\sh_d$.
\begin{defn}\label{bundlemap}
Let $p_d \colon \sh_d \MTO_d \to mO$ be the map of spectra induced by the maps of bundles
\[
\begin{tikzcd}
\xi_{V\oplus \mathbf{R}^d}^\bot \arrow{r}{\bot} \arrow{d}{} & \xi_{V\oplus \mathbf{R}^d} \arrow{r}{} \arrow{d}{} & \xi_{V\oplus \mathbf{R}^\infty} \arrow{d}{}\\
\mathrm{Gr}_d({V\oplus \mathbf{R}^d}) \arrow{r}{\bot} & \mathrm{Gr}_{|V|}({V\oplus \mathbf{R}^d}) \arrow{r}{} & \mathrm{Gr}_{|V|}({V\oplus \mathbf{R}^\infty}).
\end{tikzcd}
\]
Similarly, define spectrum maps
$j_d\colon \sh_{d} \MTO_d \to \sh_{d+1} \MTO_{d+1}$ via the bundle maps
\[
\begin{tikzcd}
    \xi_{V\oplus \mathbf{R}^{d}}^\bot \arrow[r] \arrow[d] & \xi_{V\oplus \mathbf{R}^{d+1}}^\bot \arrow[d] \\
    \mathrm{Gr}_d(V\oplus \mathbf{R}^{d}) \arrow[r, "-\oplus \mathbf{R}"] & \mathrm{Gr}_{d+1} (V\oplus \mathbf{R}^{d+1}),
\end{tikzcd}
\]
cf. \cite[3.4]{gmtw} and \cite[Proposition VI.2.12]{global}
\end{defn}

The maps $p_d$ and $j_d$ are compatible in the sense that the $p_d$ define an isomorphism from the colimit
\begin{align}\label{filtration}
\mathbb{S}\cong \MTO_0 \xrightarrow{j_0} \sh_1 \MTO_1 \xrightarrow{j_1}  \ldots \to \sh_d \MTO_d \to \ldots
\end{align}
to $mO$.  Thus the $\MTO_d$ may be viewed as the stages of a filtration on $mO$.

The main technical result in this section is
\begin{proposition}\label{sequence}
  There is a cofiber sequence of $G$-equivariant spectra
    \[
      \Sigma^{\infty+d}_{+} \mathrm{Gr}_{d+1}(\mathcal{U}) \to \sh_{d}\MTO_{d} \xrightarrow{j_{d}} \sh_{d+1} \MTO_{d+1}.
    \]
\end{proposition}
Before giving the proof, let us explain some consequences relevant for us.  Firstly, since equivariant spectra have vanishing homotopy groups in negative degrees ($\pi_i^H = 0$ for all $H < G$ and all $i < 0$), the long exact sequence in homotopy groups implies that $j_d$ induces an isomorphism
\begin{equation*}
  \pi_{i}^G(\MTO_{d}) \xrightarrow{\cong} \pi_{i-1}^G(\MTO_{d+1})
\end{equation*}
for all $i < 0$.
It also implies an exact sequence
\begin{equation*}
  \pi_0^G(\MTO_{d+1}) \to \pi_0^G(\Sigma^\infty_+ \mathrm{Gr}_{d+1}(\mathcal{U})) \to
  \pi_0^G \MTO_d \xrightarrow{(j_{d})_*} \pi_{-1}^G \MTO_{d+1} \to 0.
\end{equation*}
These observations may be combined to the following
\begin{coro}\label{mOMTO}
  The maps $p_d$ induce isomorphisms
  \[
    \pi_{i}^G \MTO_d \xrightarrow{(p_d)_*} \pi_{d+i}^G mO
  \]
  for all $i < 0$, and an exact sequence
  \begin{equation}\label{eq:13}
    \pi_0^G(\MTO_{d+1}) \to \textbf{A}(O(d+1),G) \to \pi_0^G \MTO_d \xrightarrow{(p_{d})_*} \pi_{d}^G mO \to 0,
  \end{equation}
  where $\mathbf{A}(O(d+1),G)$ is a free abelian group with generators given by $(G\times O(d+1))$-conjugacy classes $[H,\rho]$ for a subgroup $H\leq G$ and a homomorphism $\rho\colon H\to O(d+1)$. That is, $\mathbf{A}(O(d+1),G)$ has generating set 
  $$\coprod_{(H)} \textup{Rep}_{d+1}(H)/W_G H, $$
  where $\textup{Rep}_{d+1}(H)/W_G H$ are Weyl group orbits of isomorphism classes of $(d+1)$-di\-men\-sion\-al $H$-representations, and the disjoint union is over conjugacy classes of subgroups $H\leq G$.
\end{coro}
\begin{proof}
  The isomorphisms $\pi_{i}^G(\MTO_{d}) \xrightarrow{\cong} \pi_{i-1}^G(\MTO_{d+1})$ for $i < 0$ become an isomorphism $(p_d)_*: \pi_i^G(\MTO_d) \to \pi_{d+i} mO$ in the colimit.  It then remains to identify $\pi_0^G( \Sigma^\infty_+ \mathrm{Gr}_{d+1}(\mathcal{U}))$ with $\mathbf{A}(O(d+1),G)$.

  This identification is a special case of the tom Dieck splitting, cf.\ \cite[V.9 Corollary 9.3]{equi_stable_htopy}.  Firstly, for any $H < G$ the set  $\pi_0^H(\mathrm{Gr}_d(\mathcal{U})^H)$ is canonically identified with $\textup{Rep}_d(H)$, as we already used: $V \in (\mathrm{Gr}_d(\mathcal{U}))^H$ implies that the $G$-action on $\mathcal{U}$ restricts to an $H$-action on $V \subset \mathcal{U}$, defining an isomorphism class $[V] \in \textup{Rep}_d(H)$.  The tom Dieck splitting then gives
  \begin{equation*}
    (\Sigma^\infty_+ \mathrm{Gr}_d(\mathcal{U}))^G \simeq \bigvee_{H \leq G} \Sigma^\infty_+ (\mathrm{Gr}_d(\mathcal{U})^H)_{hW_G(H)},
  \end{equation*}
  where the subscript denotes homotopy orbits under the action of the Weyl group $W_G(H) = N_G(H)/H$ acting on the $H$-fixed points, and the wedge sum is over subgroups $H \leq G$, one in each conjugacy class.  By taking $\pi_0$ we get an isomorphism $\pi_0((\Sigma^\infty_+ \mathrm{Gr}_d(\mathcal{U}))^G) \cong \textbf{A}(O(d),G)$.  
\end{proof}

By the equivariant Whitney embedding theorem we also have 
$$\mathcal{N}_{d-1}^G\cong \pi_0^G B\mathcal{C}_d. $$
Hence we recover~\eqref{bordism} as a consequence of Theorem~\ref{thm:mainthm} and \cref{mOMTO}.  The sequence \eqref{eq:13} may now be written
\begin{equation*}
  \pi_1^G B\mathcal{C}_{d+1} \xrightarrow{\chi} \textbf{A}(O(d+1),G) \to \pi_1^G B\mathcal{C}_d
   \to \mathcal{N}_d^G \to 0,
\end{equation*}
where the homomorphism $\chi$ can be thought of as a kind of equivariant Euler
characteristic.

\begin{proof}[Proof of Proposition~\ref{sequence}]
    We can follow the proof of \cite[Proposition 3.1]{gmtw}. For any two $G$-equivariant vector bundles $E$ and $F$ over the same base $G$-space $B$, there is a $G$-cofiber sequence
    \begin{align}\label{cofiber}
        \mathrm{Th}(p^* E) \to \mathrm{Th}(E) \to \mathrm{Th}(E\oplus F),
    \end{align}
    where $p\colon S(F)\to B$ is the bundle projection of the sphere bundle.
    
    Apply \eqref{cofiber} to $B=\mathrm{Gr}_d(V\oplus 1)$, $F_V=\xi$ and $E_V=\xi^\bot$ (the tautological bundle and its orthogonal complement), to get the cofiber sequence of $G$-spaces
    \begin{align}\label{cofibspace}
        \mathrm{Th}(p^* E_V) \to \MTO_d(V\oplus 1) \to \Sigma^{V\oplus 1} \mathrm{Gr}_d(V\oplus 1)_+.
    \end{align}
    The Thom spaces $\mathrm{Th}(p^* E_V)$ for varying $V$ assemble into an orthogonal spectrum $\mathrm{Th}(p^*E)$ that is equivalent to $\MTO_{d-1}$ for the following reason.
    
    Consider the $G$-fiber sequence
    \begin{align}
        \mathrm{Gr}_{d-1}(V)\to   S(F_V) \to  S^{V\oplus 1},
    \end{align}
    mapping $(L,v)\in S(F_V)$ to the unit vector $v\in S^{V\oplus 1}$.
    For a subgroup $H\leq G$, the fixed point space $(S^{V\oplus 1})^H=S^{V^H\oplus 1}$ is $\dim (V^H)$-connected. This shows that the map $\mathrm{Gr}_{d-1}(V)\to   S(F_V)$ is $\dim (V^H)$-connected on $H$-fixed points. The bundle $p^* E_V$ pulls back to the complement of the tautological bunle, $\xi^\bot$ over $\mathrm{Gr}_{d-1}(V)$, hence giving a map $\MTO_{d-1}(V)\to \mathrm{Th}(p^*E_V)$ that is $(2\dim (V^H)-d)$-connected on $H$-fixed points.
    
    Thus we proved the equivalence of spectra $\MTO_{d-1}\simeq \mathrm{Th}(p^* E)$, which together with the sequence \eqref{cofibspace} implies the statement of the lemma.
\end{proof}

\begin{rmrk}[Compact Lie groups]
	Let $G$ be a compact Lie group. Our definitions in \cref{sec:defs}, and hence both sides of our main Theorem~\ref{thm:mainthm} make sense in this case as well. We have not pursued to what extent our results may generalize, but we offer the following remarks.
	
    As discussed above, the classical result on geometric equivariant bordism can be recovered by taking $\pi_0$ of both sides in Theorem~\ref{thm:mainthm}. As explained in \cite[6.2.33]{global}, \eqref{bordism} fails e.g.\ for $G = SU(2)$, showing that Theorem~\ref{thm:mainthm} cannot be true for $SU(2)$.
    
    Based on these $\pi_0$ investigations it seems possible that Theorem~\ref{thm:mainthm} could be true whenever $G$ is a product of a finite group and a torus. However, our methods do not immediately generalize to this case, in particular the unstable statement \cref{unstable_theorem} fails for $G = S^1$.  Let us briefly explain this, which follows from an example pointed out by \cite{segal}:  let $G = S^1$, $d = 0$ and let $V$ be $\mathbf{R}^3$ with $G$ acting by rotation around an axis. Then there is a $G$-cofibration sequence
    $$
    S^1 \to S^V \to S^2\wedge G_{+},
    $$
    and hence for any $G$-space $X$ a fibration sequence
    $$
    \Omega^2 X \to \left( \Omega^V X\right )^G \to \Omega X^G.
    $$
    Taking $X = S^V$ this shows that 
    $$\left ( \Omega^V S^V \right )^G \simeq \mathbf{Z}\times \Omega^2 S^3.$$
    For the representation $V$ above and taking $d = 0$ we have $\Omega B\mathcal{C}_d(V)^G\simeq \mathbf{Z}$, since $G$-equivariant configurations of points in $V$ are configurations of points on $V^G = \mathbf{R}$. This shows failure of \cref{unstable_theorem}, since by taking loops of fixed points on both sides we get
    \begin{equation*}
    \Omega B\mathcal{C}_d(V)^G\simeq \mathbf{Z} \not\simeq  \mathbf{Z}\times \Omega^2 S^3 \simeq \left ( \Omega^V S^V \right )^G\qedhere
  \end{equation*}
\end{rmrk}

\bibliographystyle{alpha}
\bibliography{cobcat_bib}

\end{document}